\documentclass{elsarticle}
\usepackage{graphicx}
\usepackage{amscd,amsmath,amstext,amsfonts,amsbsy,amssymb,amsthm,eufrak}
\usepackage{hyperref}
\usepackage[toc,page]{appendix}
\usepackage{chngcntr}

\newtheorem{theo}{Theorem}
\newtheorem{lemm}[theo]{Lemma}

\newtheorem{prop}[theo]{Proposition}

\newdefinition{defi}{Definition}

\newdefinition{rema}{Remark}

\makeatletter
\def\ps@pprintTitle{
 \let\@oddhead\@empty
 \let\@evenhead\@empty
 \def\@oddfoot{}%
 \let\@evenfoot\let\@oddfoot }
\makeatother

\begin{document}
\title{\textbf{Blow-up results for a strongly perturbed semilinear heat equation: Theoretical analysis and numerical method}}
\author{V. T. Nguyen and H. Zaag\footnote{This author is supported by the ERC Advanced Grant no. 291214, BLOWDISOL and by the ANR project ANA\'E ref. ANR-13-BS01-0010-03.}\\
\textit{\small{Universit{\'e} Paris 13, Sorbonne Paris Cit{\'e},}}\\
\textit{\small{LAGA, CNRS (UMR 7539), F-93430, Villetaneuse, France.}} }

\begin{abstract}
We consider a blow-up solution for a strongly perturbed semilinear heat equation with Sobolev subcritical power nonlinearity. Working in the framework of similarity variables, we find a Lyapunov functional for the problem. Using this Lyapunov functional, we derive the blow-up rate and the blow-up limit of the solution. We also classify all asymptotic behaviors of the solution at the singularity and give precisely blow-up profiles corresponding to these behaviors. Finally, we attain the blow-up profile numerically, thanks to a new mesh-refinement algorithm inspired by the rescaling method of Berger and Kohn \cite{BKcpam88}. Note that our method is applicable to more general equations, in particular those with no scaling invariance.\\

\end{abstract}
\maketitle
\section{Introduction}
We are concerned in this paper with blow-up phenomena arising in the following nonlinear heat problem:
\begin{equation}\label{equ:problem}
\left\{
\begin{array}{rcl}
u_t &=& \Delta u + |u|^{p-1}u + h(u), \\
u(.,0) &=& u_0 \in L^\infty(\mathbb{R}^n),
\end{array}
\right.
\end{equation}
where $u(t): x \in \mathbb{R}^n \to u(x,t) \in \mathbb{R}$ and $\Delta$ stands for the Laplacian in $\mathbb{R}^n$. The exponent $p > 1$ is subcritical (that means that $p < \frac{n + 2}{n-2}$ if $n \geq 3$) and $h$ is given by 
\begin{equation}\label{equ:h}
h(z) = \mu\dfrac{|z|^{p-1}z}{\log^{a}(2 + z^2)}, \quad \text{with}\;\;a > 0, \; \mu \in \mathbb{R}.
\end{equation}

\noindent By standard results, the problem \eqref{equ:problem} has a unique classical solution $u(x,t)$ in $L^\infty(\mathbb{R}^n)$, which exists at least for small times. The solution $u(x,t)$ may develop singularities in some finite time. We say that $u(x,t)$ blows up in a finite time $T$ if $u(x,t)$ satisfies \eqref{equ:problem} in $\mathbb{R}^b \times [0,T)$ and 
$$\lim_{t \to T} \|u(t)\|_{L^\infty(\mathbb{R}^n)} = +\infty.$$
$T$ is called the blow-up time of $u(x,t)$. In such a blow-up case, a point $b \in \mathbb{R}^n$ is called a blow-up point of $u(x,t)$ if and only if there exist $(x_n, t_n) \to (b, T)$ such that $|u(x_n,t_n)| \to +\infty$ as $n \to +\infty$.\\

\noindent In the case $\mu = 0$, the equation \eqref{equ:problem} is the semilinear heat equation, 
\begin{equation}\label{equ:semi}
u_t = \Delta u + |u|^{p-1}u.
\end{equation}
Problem \eqref{equ:semi} has been addressed in different ways in the literature. The existence of blow-up solutions has been proved by several authors (see Fujita \cite{FUJsut66},  Levine \cite{LEVarma73}, Ball \cite{BALjmo77}). Consider $u(x,t)$ a solution of \eqref{equ:semi} which blows up at a time $T$. The very first question to be answered is the blow-up rate, i.e. there are positive constants $C_1, C_2$ such that
\begin{equation}\label{equ:blrateIn}
C_1(T -t)^{-\frac{1}{p-1}} \leq \|u(t)\|_{L^\infty(\mathbb{R}^n)} \leq C_2(T -t)^{-\frac{1}{p-1}}, \quad \forall t \in (0, T).
\end{equation}
The lower bound in \eqref{equ:blrateIn} follows by a simple argument based on Duhamel's formula (see Weissler \cite{We81ijm}). For the upper bound, Giga and Kohn proved \eqref{equ:blrateIn} in \cite{GKiumj87} for $1 < p < \frac{3n + 8}{3n-4}$ or for non-negative initial data with subcritical $p$. Then, this result was extended to all subcritiacal $p$ without assuming non-negativity for initial data $u_0$ by Giga, Matsui and Sasayama in \cite{GMSiumj04}. The estimate \eqref{equ:blrateIn} is a fundamental step to obtain more information about the asymptotic blow-up behavior, locally near a given blow-up point $\hat{b}$. Giga and Kohn showed in \cite{GKcpam89} that for a given blow-up point $\hat{b} \in \mathbb{R}^n$,
$$\lim_{t \to T} (T-t)^\frac{1}{p-1}u(\hat{b} + y\sqrt{T-t},t) = \pm\kappa,$$
where $\kappa = (p-1)^{-\frac{1}{p-1}}$, uniformly on compact sets of $\mathbb{R}^n$.\\
This result was specified by  Filippas ans Liu \cite{FLaihn93} (see also Filippas and Kohn \cite{FKcpam92}) and Vel\'azquez \cite{VELcpde92}, \cite{VELtams93} (see also Herrero and Vel\'azquez \cite{HVdie92}, \cite{HVaihn93}, \cite{HVcpde92}). Using the renormalization theory,  Bricmont and Kupiainen showed in \cite{BKnon94} the existence of a solution of \eqref{equ:semi} such that 
\begin{equation}\label{equ:limpro}
\left\|(T-t)^\frac{1}{p-1}u(\hat{b} + z\sqrt{(T-t)|\log(T-t)|},t) - f_0(z)\right\|_{L^\infty(\mathbb{R}^n)} \to 0 \quad \text{as} \quad t \to T,
\end{equation}
where 
\begin{equation}\label{def:f0}
f_0(z) = \kappa\left(1 + \frac{p-1}{4p}|z|^2\right)^{-\frac{1}{p-1}}.
\end{equation}
Merle and Zaag in \cite{MZdm97} obtained the same result through a reduction to a finite dimensional problem. Moreover, they showed that the profile \eqref{def:f0} is stable under perturbations of initial data (see also \cite{FMZma00}, \cite{FZnon00} and \cite{MZjfa08}).\\ 

\noindent In the case where the function $h$ satisfies
\begin{equation}\label{equ:h1}
j = 0, 1, \;\; |h^{(j)}(z)| \leq M\left(\frac{|z|^{p-j}}{\log^a(2 + z^2)} + 1 \right), \quad  |h''(z)| \leq M\frac{|z|^{p-2}}{\log^a(2 + z^2)},
\end{equation}
with $a > 1$ and $M > 0$, we proved in \cite{NG14the} the existence of a Lyapunov functional in \textit{similarity variables} for the problem \eqref{equ:problem} which is a crucial step in deriving the estimate \eqref{equ:blrateIn}. We also gave a classification of possible blow-up behaviors of the solution when it approaches to singularity. In \cite{NG14the2}, we constructed a blow-up solution of the problem \eqref{equ:problem} satisfying the behavior described in \eqref{equ:limpro} in the case where $h$ satisfies the first estimate in \eqref{equ:h1} or $h$ is given by \eqref{equ:h}. \\

In this paper, we aim at extending the results of \cite{NG14the} to the case $a \in (0,1]$. As we mentioned above, the first step is to derive the blow-up rate of the blow-up solution. As in \cite{GMSiumj04} and \cite{NG14the}, the key step is to find a Lyapunov functional in \textit{similarity variables} for equation \eqref{equ:problem}. More precisely, we introduce for all $b \in \mathbb{R}^n$ ($b$ may be a blow-up point of $u$ or not) the following \emph{similarity variables}:
\begin{equation}\label{equ:simivariables}
y = \frac{x-b}{\sqrt{T-t}}, \quad s = -\log(T-t), \quad w_{b,T} = (T - t)^\frac{1}{p-1}u(x,t).
\end{equation}
Hence $w_{b,T}$ satisfies for all $s \geq -\log{T}$ and for all $y \in \mathbb{R}^n$:
\begin{equation}\label{equ:divw1}
\partial_sw_{b,T} = \frac{1}{\rho}\text{div}(\rho \nabla w_{b,T}) - \frac{w_{b,T}}{p-1}  + |w_{b,T}|^{p-1}w_{b,T} + e^{-\frac{ps}{p-1}}h\left(e^{\frac{s}{p-1}}w_{b,T}\right),
\end{equation}
where 
\begin{equation}\label{equ:rho}
\rho(y) = \left(\frac{1}{4\pi}\right)^{n/2}e^{-\frac{|y|^2}{4}}.
\end{equation}
Following the method introduced by Hamza and Zaag in \cite{HZjhde12}, \cite{HZnonl12} for perturbations of the semilinear wave equation, we introduce
\begin{equation}\label{equ:lyafun}
\mathcal{J}_a[w](s) = \mathcal{E}[w](s)e^{\frac{\gamma}{a}s^{-a}} + \theta s^{-a},
\end{equation}
where $\gamma$, $\theta$ are positive constants depending only on $p$, $a$, $\mu$ and $n$ which will be determined later, and
\begin{equation}\label{equ:difE}
\mathcal{E}[w] = \mathcal{E}_0[w] + \mathcal{I}[w],
\end{equation}
where 
\begin{equation}\label{equ:E0}
\mathcal{E}_0[w](s) = \int_{\mathbb{R}^n} \left(\frac{1}{2}|\nabla w|^2  + \frac{1}{2(p-1)}|w|^2 - \frac{1}{p+1}|w|^{p+1}\right)\rho dy,
\end{equation}
and
\begin{equation}\label{def:I_H}
\mathcal{I}[w](s) = - e^{-\frac{p + 1}{p-1}s}\int_{\mathbb{R}^n} H\left(e^{\frac{s}{p-1}}w\right)\rho dy, \;\; H(z) = \int_0^z h(\xi)d\xi.
\end{equation}
The main novelty of this paper is to allow values of $a$ in $(0,1]$, and this is possible at the expense of taking the particular form \eqref{equ:h} for the perturbation $h$. We aim at the following:
\begin{theo}[\textbf{Existence of a Lyapunov functional for equation \eqref{equ:divw1}}] \label{theo:lya} Let $a, p, n, \mu$ be fixed, consider $w$ a solution of equation \eqref{equ:divw1}. Then, there exist $\hat{s}_0 = \hat{s}_0(a, p, n, \mu) \geq s_0$, $\hat{\theta}_0 = \hat{\theta}_0(a,p,n,\mu)$ and $\gamma = \gamma(a,p,n,\mu)$ such that if $\theta \geq \hat{\theta}_0$, then $\mathcal{J}_a$ satisfies the following inequality, for all $s_2 > s_1 \geq \max\{\hat{s}_0, -\log T\}$, 
\begin{equation}\label{equ:estimateJinT}
\mathcal{J}_a[w](s_2) - \mathcal{J}_a[w](s_1) \leq - \frac{1}{2}\int_{s_1}^{s_2}\int_{\mathbb{R}^n}(\partial_sw)^2\rho dy ds.
\end{equation}
\end{theo}

\noindent As in \cite{GMSiumj04} and \cite{NG14the}, the existence of the Lyapunov functional is a crucial step for deriving the blow-up rate \eqref{equ:blrateIn} and then the blow-up limit. In particular, we have the following:
\begin{theo} \label{theo:blrate} Let $a, p, n, \mu$ be fixed and $u$ be a blow-up solution of equation \eqref{equ:problem} with a blow-up time $T$.\\ 
$(i)\;$ (\textbf{Blow-up rate}) There exists $\hat{s}_1=\hat{s}_1(a, p, n, \mu) \geq \hat{s}_0$ such that for all $s \geq s' = \max\{\hat{s}_1, -\log T\}$, 
\begin{equation}\label{equ:boundw}
\|w_{b,T}(y,s)\|_{L^\infty(\mathbb{R}^n)} \leq C,
\end{equation}
where $w_{b,T}$ is defined in \eqref{equ:simivariables} and $C$ is a positive constant depending only on $n, p, \mu$ and a bound of $\|w_{b,T}(\hat{s}_0)\|_{L^\infty}$.\\
$(ii)\;$ (\textbf{Blow-up limit}) If $\hat{a}$ is a blow-up point, then
\begin{equation}\label{equ:limitw}
\lim_{t \to T}(T-t)^{\frac{1}{p-1}}u(\hat{a} + y\sqrt{T-t},t) = \lim_{s \to +\infty}\, w_{\hat{a},T} (y,s) = \pm\kappa,
\end{equation}
holds in $L^2_\rho$ ($L^2_\rho$ is the weighted $L^2$ space associated with the weight $\rho$ \eqref{equ:rho}), and also uniformly on each compact subset of $\mathbb{R}^n$.
\end{theo}
\begin{rema} We will not give the proof of Theorem \ref{theo:blrate} because its proof follows from Theorem \ref{theo:lya} as in \cite{NG14the}. Hence, we only give the proof of Theorem \ref{theo:lya} and refer the reader to Section 2 in \cite{NG14the} for the proofs of \eqref{equ:boundw} and \eqref{equ:limitw} respectively. 
\end{rema}

\noindent The next step consists in obtaining an additional term in the asymptotic expansion given in $(ii)$ of Theorem \ref{theo:blrate}. Given $b$ a blow-up point of $u(x,t)$, and up to changing $u_0$ by $-u_0$ and $h$ by $-h$, we may assume that $w_{b,T} \to \kappa$ in $L^2_\rho$ as $s \to +\infty$. As in \cite{NG14the}, we linearize $w_{b,T}$ around $\phi$, where $\phi$ is the positive solution of the ordinary differential equation  associated to \eqref{equ:divw1},
\begin{equation}\label{equ:phiODE}
\phi_s = -\frac{\phi}{p-1} + \phi^p + e^{-\frac{ps}{p-1}}h\left(e^\frac{s}{p-1}\phi\right)
\end{equation}
such that 
\begin{equation}\label{equ:solphi}
\phi(s) \to \kappa \quad \text{as} \quad s \to + \infty,
\end{equation}
(see Lemma A.3 in \cite{NG14the} for the existence of $\phi$, and note that $\phi$ is unique. For the reader's convenience, we give in Lemma \ref{ap:lemmA3} the expansion of $\phi$ as $s \to +\infty$).\\

\noindent Let us introduce $v_{b,T} = w_{b,T} - \phi(s)$, then $\|v_{b,T}(y,s)\|_{L^2_\rho} \to 0$ as $s \to +\infty$ and $v_{b,T}$ (or $v$ for simplicity) satisfies the following equation:
\begin{equation*}
\partial_s v = (\mathcal{L} + \omega(s))v + F(v) + H(v,s),\quad \forall y \in \mathbb{R}^n, \; \forall s \in [-\log T, +\infty),
\end{equation*}
where $\mathcal{L} = \Delta - \frac{y}{2}\cdot \nabla + 1$ and $\omega$, $F$, $H$ satisfy
$$\omega(s) = \mathcal{O}(\frac{1}{s^{a + 1}})\quad\text{and} \quad |F(v)| + |H(v,s)| = \mathcal{O}(|v|^2) \quad \text{as $s \to +\infty$},$$
(see the beginning of Section \ref{sec:refasy} for the proper definitions of $\omega$, $F$ and $G$).\\
It is well known that the operator $\mathcal{L}$ is self-adjoint in $L^2_\rho(\mathbb{R}^n)$. Its spectrum  is given by
$$spec(\mathcal{L}) = \{1 - \frac{m}{2},\; m \in \mathbb{N}\},$$
and it consists of eigenvalues. The eigenfunctions of $\mathcal{L}$ are derived from Hermite polynomials:\\
- For $n = 1$, the eigenfunction corresponding to $1 - \frac{m}{2}$ is
\begin{equation}\label{equ:eigenfu1}
h_m(y) = \sum_{k= 0}^{\left[\frac{m}{2}\right]} \frac{m!}{k!(m - 2k)!}(-1)^ky^{m - 2k}, 
\end{equation}
- For $n \geq 2$: we write the spectrum of $\mathcal{L}$ as 
$$spec(\mathcal{L}) = \{ 1 - \frac{|m|}{2},\; |m| = m_1 + \dots + m_n, \;(m_1,\dots, m_n) \in \mathbb{N}^n\}.$$
For $m = (m_1, \dots, m_n) \in \mathbb{N}^n$, the eigenfunction corresponding to $1 - \frac{|m|}{2}$ is 
\begin{equation}\label{equ:eigenfu}
H_m(y) = h_{m_1}(y_1)\dots h_{m_n}(y_n),
\end{equation}
where $h_{m}$ is defined in \eqref{equ:eigenfu1}.\\
We also denote $c_m = c_{m_1}c_{m_2}\dots c_{m_n}$ and $y^m = y_1^{m_1}y_2^{m_2}\dots y_n^{m_n}$ for any $m = (m_1,\dots, m_n) \in \mathbb{N}^n$ and $y = (y_1, \dots, y_n) \in \mathbb{R}^n$.\\
By this way, we derive the following  asymptotic behaviors of $w_{b,T}(y,s)$ as $s \to +\infty$:
\begin{theo}[\textbf{Classification of the behavior of $w_{b,T}$ as $s \to +\infty$}]\label{theo:refinedasymptotic} Consider $u(t)$ a solution of equation \eqref{equ:problem} which blows-up at time $T$ and $b$ a blow-up point. Let $w_{b,T}(y,s)$ be a solution of equation \eqref{equ:divw1}. Then one of the following possibilities occurs:\\
$i)\;$ $w_{b,T}(y,s) \equiv \phi(s)$,\\
$ii)$ There exists $l \in \{1, \dots, n\}$ such that up to an orthogonal  transformation of coordinates, we have
$$
w_{b,T}(y,s) = \phi(s) -\frac{\kappa}{4ps} \left(\sum_{j=1}^l y_j^2 -2 l\right) + \mathcal{O}\left(\frac{1}{s^{a+1}}\right)+ \mathcal{O}\left(\frac{\log s}{s^2}\right) \quad \text{as} \quad s \to +\infty,
$$
$iii)$ There exist an integer number $m \geq 3$ and constants $c_\alpha$ not all zero such that 
$$w_{b,T}(y,s) = \phi(s) - e^{-\left(\frac{m}{2}-1\right)s}\sum_{|\alpha| = m}c_\alpha H_\alpha(y) + o\left(e^{-\left(\frac{m}{2}-1\right)s}\right) \quad \text{as}\quad s \to +\infty. $$
The convergence takes place in $L^2_\rho$ as well as in $\mathcal{C}^{k,\gamma}_{loc}$ for any $k \geq 1$ and some $\gamma \in (0,1)$.
\end{theo}
\begin{rema} In our previous paper \cite{NG14the}, we were unable to get this result in the case where $h$ satisfies \eqref{equ:h1} with $a \in (0,1]$. Here, by taking the particular form of the perturbation (see  \eqref{equ:h}), we are able to overcome technical difficulties in order to derive the result.
\end{rema}
\begin{rema} From $ii)$ of Theorem \ref{theo:blrate}, we would naturally 
try to find an equivalent for $w- \kappa$ as $s \to +\infty$. A posteriori from our results in Theorem \ref{theo:refinedasymptotic}, we see that in all cases $\|w - \kappa\|_{L^2_\rho} \sim \frac{C}{s^{a'}}$ with $a' = \max\{a,1\}$. This is indeed a new phenomenon observed in our equation \eqref{equ:problem}, and which is different from the case of the unperturbed semilinear heat equation where either $w - \kappa \equiv 0$, or $\|w - \kappa\|_{L^2_\rho} \sim \frac{C}{s}$ or $\|w - \kappa\|_{L^2_\rho} \sim Ce^{(1 - m/2)s}$ for some even $m \geq 4$. This shows the originality of our paper. In our case, linearizing around $\kappa$ would keep us trapped in the $\frac{1}{s}$ scale. In order to escape that scale, we forget the explicit function $\kappa$ which is not a solution of equation \eqref{equ:divw1}, and linearizing instead around the non-explicit function $\phi$, which happens to be an exact solution of \eqref{equ:divw1}. This way, we escape the $\frac{1}{s}$ scale and reach exponentially decreasing order.
\end{rema}

\noindent Using the information obtained in Theorem \ref{theo:refinedasymptotic}, we can extend the asymptotic behavior of $w_{b,T}$ to larger regions. Particularly, we have the following:
\begin{theo}[\textbf{Convergence extension of $w_{b,T}$ to larger regions}] \label{theo:pro} For all $K_0 > 0$,\\
$i)$ if $ii)$ of Theorem \ref{theo:refinedasymptotic} occurs, then
\begin{equation}\label{equ:prof1stab}
\sup_{|\xi| \leq K_0} \left|w_{b,T}(\xi \sqrt{s},s) - f_l(\xi)\right| = \mathcal{O}\left(\frac{1}{s^{a}}\right) + \mathcal{O}\left(\frac{\log s}{s}\right),\quad \text{as} \quad s \to +\infty, 
\end{equation}
where  
\begin{equation}\label{def:f}
\forall \xi \in \mathbb{R}^n, \quad f_l(\xi) = \kappa\left(1 + \frac{p-1}{4p} \sum_{j=1}^l \xi_j^2 \right)^{-\frac{1}{p-1}},
\end{equation}
with $l$ given in $ii)$ of Theorem \ref{theo:refinedasymptotic}.\\
\noindent $ii)$ if $iii)$ of Theorem \ref{theo:refinedasymptotic} occurs, then $m \geq 4$ is even, and 
\begin{equation}\label{equ:prof2unstab}
\sup_{|\xi| \leq K_0} \left|w_{b,T}\left(\xi e^{\left(\frac{1}{2} - \frac{1}{m}\right)s}\right) - \psi_m(\xi)\right| \to 0  \quad \text{as} \quad s \to +\infty,
\end{equation}
where 
\begin{equation}
\forall \xi \in \mathbb{R}^n, \quad \psi_m(\xi) = \kappa\left(1 + \kappa^{-p} \sum_{|\alpha| = m}c_\alpha\xi^\alpha\right)^{-\frac{1}{p-1}},
\end{equation}
with $c_\alpha$ the same as in Theorem \ref{theo:refinedasymptotic}, and the multilinear for $\sum_{|\alpha| = m}c_\alpha\xi^\alpha$ is nonnegative.
\end{theo}
\begin{rema} As in the unperturbed case ($h \equiv 0$), we expect that \eqref{equ:prof1stab} is stable (see the previous remarks, particularly the paragraph after \eqref{equ:limpro}), and \eqref{equ:prof2unstab} should correspond to unstable behaviors (the unstable of \eqref{equ:prof2unstab} was proved only in one space dimension by Herrero and Vel\'azquez in \cite{HVasnsp92} and \cite{HVasps92}). While remarking numerical simulation for equation \eqref{equ:problem} in one space dimension (see Section \ref{sec:4} below), we see that the numerical solutions exhibit only the behavior \eqref{equ:prof1stab}, we could never obtain the behavior \eqref{equ:prof2unstab}. This is probably due to the fact that the behavior \eqref{equ:prof2unstab} is unstable.
\end{rema}

At the end of this work, we give numerical confirmations for the asymptotic profile described in Theorem \ref{theo:pro}. For this purpose, we propose a new mesh-refinement method inspired by the rescaling algorithm of Berger and Kohn \cite{BKcpam88}. Note that, their method was successful to solve blowing-up problems which are invariant under the following transformation,
\begin{equation}\label{equ:invE}
\forall \gamma > 0, \quad \gamma \mapsto u_\gamma (\xi, \tau) = \gamma^{\frac{2}{p-1}}u(\gamma \xi, \gamma^2 \tau).
\end{equation}
However, there are a lot of equations whose solutions blow up in finite time but the equation does not satisfy the property \eqref{equ:invE}, one of them is the equation \eqref{equ:problem} because of the presence of the perturbation term $h$. Although our method is very similar to Berger and Kohn's algorithm in spirit, it is better in the sense that it can be applied to a larger class of blowing-up problems which do not satisfy the rescaling property \eqref{equ:invE}. Up to our knowledge, there
are not many papers on the numerical blow-up profile, apart from the paper of Berger and Kohn \cite{BKcpam88} (see also \cite{NG14num}), who already obtained numerical results for equation \eqref{equ:problem} without the perturbation term. For other numerical aspects, there are several studies for \eqref{equ:problem} in the unperturbed case, see for example, Abia, L{\'o}pez-Marcos and Mart{\'i}nez in \cite{ALManm01}, \citep{ALManm98}, Groisman and Rossi \cite{GRjcam01},\cite{GRaa04}, \cite{GROcomp06}, N'gohisse and Boni \cite{NBams11}, Kyza and Makridakis \cite{KMsiam11}, Cangiani et al. \cite{CGKM14} and the references therein. There is also the work of Baruch et al. \cite{BFGpd10} studying standing-ring solutions. \\

\noindent This paper is organized as follows: Section \ref{sec:Lya} is devoted to the proof of Theorem \ref{theo:lya}. Theorem \ref{theo:blrate} follows from Theorem \ref{theo:lya}. Since all the arguments presented \cite{NG14the} remain valid for the case \eqref{equ:h1}, except the existence of the Lyapunov functional for equation \eqref{equ:divw1} (Theorem \ref{theo:lya}), we kindly refer the  reader to Section 2.3 and 2.4 in \cite{NG14the} for details of the proof.  Section \ref{sec:refasy} deals with results on asymptotic behaviors (Theorem \ref{theo:refinedasymptotic} and Theorem \ref{theo:pro}). In Section \ref{sec:numme}, we describe the new mesh-refinement method and give some numerical justifications for the theoretical results.\\

\noindent \textbf{Acknowledgement:} The authors are grateful to M. A. Hamza for several helpful conversations pertaining to this work, and especially for giving the idea for the proof of Theorem \ref{theo:lya} in this paper.

\section{Existence of a Lyapunov functional for equation \eqref{equ:divw1}}\label{sec:Lya}
In this section, we mainly aim at proving that the functional $\mathcal{J}_a$ defined in \eqref{equ:lyafun} is a Lyapunov functional for equation \eqref{equ:divw1} (Theorem \ref{theo:lya}). Note that this functional is far from being trivial and makes our main contribution.\\

\noindent In what follows, we denote by $C$ a generic constant depending only on $a$, $p$, $n$ and $\mu$. We first give the following estimates on the perturbation term appearing in equation \eqref{equ:divw1}:
\begin{lemm}\label{lemm:esth} Let $h$ be the function defined in \eqref{equ:h}. For all $\epsilon \in (0,p]$, there exists $C_0 = C_0(a,\mu, p, \epsilon) > 0$  and $\bar{s}_0 = \bar{s}_0(a,p,\epsilon) > 0$ large enough such that for all $s \geq \bar{s}_0$, \\
$i)$ 
\begin{align*}
\left|e^{-\frac{ps}{p-1}} h\left(e^\frac{s}{p-1}z\right) \right| &\leq \frac{C_0}{s^a}(|z|^p + |z|^{p-\epsilon}),\\
\text{and}\;\; \left|e^{-\frac{(p+1)s}{p-1}} H\left(e^\frac{s}{p-1}z\right) \right| &\leq \frac{C_0}{s^a}(|z|^{p+1} + 1), 
\end{align*}
where $H$ is defined in \eqref{def:I_H}.\\
$ii)$ 
\begin{equation*}
\left|(p+1)e^{-\frac{(p+1)s}{p-1}} H\left(e^\frac{s}{p-1}z\right) -  e^{-\frac{ps}{p-1}} h\left(e^\frac{s}{p-1}z\right)z\right| \leq \frac{C_0}{s^{a+1}}(|z|^{p+1} + 1).
\end{equation*}
\end{lemm}
\begin{proof} Note that $i)$ obviously follows from the following estimate,
\begin{equation}\label{equ:estBah}
\forall q > 0, b > 0, \quad \frac{|z|^q}{\log^b(2 + e^\frac{2s}{p-1}z^2)} \leq \frac{C}{s^b}(|z|^q + 1), \quad \forall s \geq \bar{s}_0,
\end{equation}
where $C = C(b, q) > 0$ and $\bar{s}_0 = \bar{s}_0(b ,q) > 0$.\\
In order to derive estimate \eqref{equ:estBah}, considering the first case $z^2e^{\frac{s}{p-1}} \geq 4$, then the case $z^2e^{\frac{s}{p-1}} \leq 4$, we would obtain \eqref{equ:estBah}.\\
$ii)$ directly follows from an integration by part and estimate \eqref{equ:estBah}. Indeed, we have 
\begin{align*}
H(\xi) &= \int_0^\xi h(x)dx = \mu \int_0^\xi \frac{|x|^{p-1}x}{\log^a(2 + x^2)}dx\\
&= \frac{\mu|\xi|^{p+1}}{(p+1)\log^a(2 + \xi^2)} + \frac{2a\mu}{p+1}\int_0^\xi \frac{|x|^{p+1}x}{(2 + x^2)\log^{a + 1}(2 + x^2)}dx.
\end{align*}
Replacing $\xi$ by $e^\frac{s}{p-1}z$ and using \eqref{equ:estBah}, we then derive $ii)$. This ends the proof of Lemma \ref{lemm:esth}.
\end{proof}
\noindent We assert that Theorem \ref{theo:lya} is a direct consequence of the following lemma:
\begin{lemm} \label{lemm:lya} Let $a, p, n, \mu$ be fixed and $w$ be solution of equation \eqref{equ:divw1}. There exists $\tilde{s}_0 = \tilde{s}_0(a, p, n, \mu) \geq s_0$ such that the functional of $\mathcal{E}$ defined in \eqref{equ:difE} satisfies  the following inequality, for all $s \geq \max\{\tilde{s}_0, -\log T\}$, 
\begin{equation}\label{equ:estimateDE}
\frac{d}{ds}\mathcal{E}[w](s) \leq - \frac{1}{2}\int_{\mathbb{R}^n}w_s^2\rho dy + \gamma s^{-(a+1)}\mathcal{E}[w](s) + Cs^{-(a+1)},
\end{equation}
where $\gamma = \frac{4C_0(p+1)}{(p-1)^2}$, $C_0$ is given in Lemma \ref{lemm:esth}. 
\end{lemm}
\noindent Let us first derive Theorem \ref{theo:lya} from Lemma \ref{lemm:lya} and we will prove it later.
\begin{proof}[\textbf{Proof of Theorem \ref{theo:lya} admitting Lemma \ref{lemm:lya}}]
Differentiating the functional $\mathcal{J}$ defined in \eqref{equ:lyafun}, we obtain 
\begin{align*}
\frac{d}{ds}\mathcal{J}_a[w](s) &= \frac{d}{ds}\left\{\mathcal{E}[w](s)e^{\frac{\gamma}{a}s^{-a}} + \theta s^{-a}\right\}\\
& = \frac{d}{ds}\mathcal{E}[w](s)e^{\frac{\gamma}{a}s^{-a}} - \gamma s^{-(a+1)}\mathcal{E}[w](s)e^{\frac{\gamma}{a}s^{-a}} - a\theta s^{-(a+1)}\\
& \leq - \frac{1}{2} e^{\frac{\gamma}{a}s^{-a}}\int_{\mathbb{R}^n}w_s^2\rho dy  + \left[C e^{\frac{\gamma}{a}s^{-a}} - a\theta\right]s^{-(a+1)} \quad \text{(use \eqref{equ:estimateDE}).}
\end{align*}
Choosing $\theta$ large enough such that $C e^{\frac{\gamma}{a}{\tilde{s}_0}^{-a}} - a\theta \leq 0$ and noticing that $e^{\frac{\gamma}{a}s^{-a}} \geq 1$ for all $s > 0$, we derive 
$$
\frac{d}{ds}\mathcal{J}_a[w](s) \leq -\frac{1}{2} \int_{\mathbb{R}^n}w_s^2\rho dy, \quad \forall s \geq \tilde{s}_0.
$$
This implies inequality \eqref{equ:estimateJinT} and concludes the proof of Theorem \ref{theo:lya}, assuming that Lemma \ref{lemm:lya} holds.
\end{proof}
\noindent It remains to prove Lemma \ref{lemm:lya} in order to conclude the proof of Theorem \ref{theo:lya}.
\begin{proof}[\textbf{Proof of Lemma \ref{lemm:lya} }] 
Multiplying equation \eqref{equ:divw1} with $w_s\rho$ and integrating by parts:
\begin{align*}
\int_{\mathbb{R}^n} |w_s|^2\rho = -\frac{d}{ds} \left\{\int_{\mathbb{R}^n} \left( \frac{1}{2}|\nabla w|^2 + \frac{1}{2(p-1)} |w|^2 - \frac{1}{p+1}|w|^{p+1}\right)\rho dy \right\}&\\
+ e^{-\frac{ps}{p-1}}\int_{\mathbb{R}^n} h\left(e^{\frac{s}{p-1}}w\right)w_s\rho dy&.
\end{align*}
\noindent For the last term of the above expression, we write in the following:
\begin{align*}
e^{-\frac{ps}{p-1}}\int_{\mathbb{R}^n} h\left(e^{\frac{s}{p-1}}w\right)w_s\rho dy = e^{-\frac{(p+1)s}{p-1}} \int_{\mathbb{R}^n} h\left(e^{\frac{s}{p-1}}w\right)\left(e^{\frac{s}{p-1}} w_s + \frac{e^{\frac{s}{p-1}}}{p-1}w \right)\rho dy&\\
- \frac{1}{p-1}e^{-\frac{ps}{p-1}}\int_{\mathbb{R}^n} h\left(e^{\frac{s}{p-1}}w\right)w\rho dy&\\
= e^{-\frac{p+1}{p-1}s} \frac{d}{ds} \int_{\mathbb{R}^n} H\left(e^{\frac{s}{p-1}}w\right)\rho dy- \frac{1}{p-1} e^{-\frac{ps}{p-1}}\int_{\mathbb{R}^n} h\left(e^{\frac{s}{p-1}}w\right)w\rho dy&.
\end{align*}
This yields 
\begin{align*}
\int_{\mathbb{R}^n} |w_s|^2\rho dy = -\frac{d}{ds} \left\{\int_{\mathbb{R}^n} \left( \frac{1}{2}|\nabla w|^2 + \frac{1}{2(p-1)} |w|^2 - \frac{1}{p+1} |w|^{p+1}\right)\rho dy \right\}&\\
+ \frac{d}{ds} \left\{e^{-\frac{p + 1}{p-1}s}\int_{\mathbb{R}^n} H\left(e^{\frac{s}{p-1}}w\right)\rho dy\right\}&\\
+ \frac{p+1}{p-1}e^{-\frac{p + 1}{p-1}s}\int_{\mathbb{R}^n} H\left(e^{\frac{s}{p-1}}w\right)\rho dy&\\
- \frac{1}{p-1} e^{-\frac{ps}{p-1}}\int_{\mathbb{R}^n} h\left(e^{\frac{s}{p-1}}w\right)w\rho dy&. 
\end{align*}
From the definition of the functional $\mathcal{E}$ given in \eqref{equ:difE}, we derive a first identity in the following: 
\begin{align}
\frac{d}{ds}\mathcal{E}[w](s) = -\int_{\mathbb{R}^n} |w_s|^2\rho dy + \frac{p+1}{p-1}e^{-\frac{p + 1}{p-1}s}\int_{\mathbb{R}^n} H\left(e^{\frac{s}{p-1}}w\right)\rho dy &\nonumber\\
 - \frac{1}{p-1} e^{-\frac{ps}{p-1}}\int_{\mathbb{R}^n} h\left(e^{\frac{s}{p-1}}w\right)w\rho dy&. \label{equ:Id1}
\end{align}
A second identity is obtained by  multiplying equation \eqref{equ:divw1} with $w\rho$ and integrating by parts:
\begin{align*}
\frac{d}{ds}\int_{\mathbb{R}^n} |w|^2\rho dy  = -4\left\{\int_{\mathbb{R}^n} \left( \frac{1}{2}|\nabla w|^2 + \frac{1}{2(p-1)} |w|^2 - \frac{1}{p+1} |w|^{p+1}\right) \rho dy \right.&\\
\left.- e^{-\frac{(p + 1)s}{p-1}}\int_{\mathbb{R}^n} H\left(e^{\frac{s}{p-1}}w\right)\rho dy\right\}&\\
+ \left(2 - \frac{4}{p+1}\right)\int_{\mathbb{R}^n} |w|^{p+1}\rho dy - 4e^{-\frac{p + 1}{p-1}s}\int_{\mathbb{R}^n} H\left(e^{\frac{s}{p-1}}w\right)\rho dy &\\
+ 2e^{-\frac{ps}{p-1}}\int_{\mathbb{R}^n} h\left(e^{\frac{s}{p-1}}w\right)w\rho dy&.
\end{align*}
Using again the definition of $\mathcal{E}$ given in \eqref{equ:difE}, we rewrite the second identity in the following:
\begin{align}
\frac{d}{ds} \int_{\mathbb{R}^n} |w|^2\rho dy & = -4 \mathcal{E}[w](s) + 2\frac{p-1}{p+1}\int_{\mathbb{R}^n} |w|^{p+1}\rho dy\nonumber\\
&-  4e^{-\frac{p + 1}{p-1}s}\int_{\mathbb{R}^n} H\left(e^{\frac{s}{p-1}}w\right)\rho dy + 2e^{-\frac{ps}{p-1}}\int_{\mathbb{R}^n} h\left(e^{\frac{s}{p-1}}w\right)w\rho dy.\label{equ:Id2}
\end{align}
From \eqref{equ:Id1}, we estimate
\begin{align*}
\frac{d}{ds}\mathcal{E}[w](s) &\leq -\int_{\mathbb{R}^n} |w_s|^2\rho dy\\
& + \frac{1}{p-1}\int_{\mathbb{R}^n}\left\{ \left|(p+1)e^{-\frac{(p+1)s}{p-1}} H\left(e^{\frac{s}{p-1}}w\right) -  e^{-\frac{ps}{p-1}}h\left(e^{\frac{s}{p-1}}w\right)w \right| \right\}\rho dy. 
\end{align*}
Using $ii)$ of Lemma \ref{lemm:esth}, we have for all $s \geq \bar{s}_0$,
\begin{equation}\label{equ:estDE3}
\frac{d}{ds}\mathcal{E}[w](s) \leq -\int_{\mathbb{R}^n} |w_s|^2\rho dy + \frac{C_0 s^{-(a+1)}}{p-1}\int_{\mathbb{R}^n} |w|^{p+1}\rho dy + Cs^{-(a+1)}.
\end{equation}
On the other hand,  we have by \eqref{equ:Id2},
\begin{align*}
\int_{\mathbb{R}^n} |w|^{p+1}\rho dy &\leq \frac{2(p+1)}{p-1}\mathcal{E}[w](s) + \frac{p+1}{p-1}\int_{\mathbb{R}^n} |w_s w| \rho dy \\
& \quad +   \frac{2(p+1)}{p-1}\int_{\mathbb{R}^n}  \left( \left| e^{-\frac{p + 1}{p-1}s} H\left(e^{\frac{s}{p-1}}w\right) \right|+ \left | e^{-\frac{ps}{p-1}} h\left(e^{\frac{s}{p-1}}  w\right)w \right| \rho dy \right).
\end{align*}
Using the fact that $|w_sw| \leq \epsilon (|w_s|^2 + |w|^{p+1}) + C(\epsilon)$ for all $\epsilon > 0$ and $i)$ of Lemma \ref{lemm:esth}, we obtain 
\begin{align*}
\int_{\mathbb{R}^n} |w|^{p+1}\rho dy &\leq  \frac{2(p+1)}{p-1}\mathcal{E}[w](s) + \epsilon \int_{\mathbb{R}^n} |w_s|^2\rho dy  \\
& \quad + \left(\epsilon + Cs^{-a}\right)\int_{\mathbb{R}^n} |w|^{p+1}\rho dy + C.
\end{align*}
Taking $\epsilon = \frac 14$ and $s_1$ large enough such that $Cs^{-a} \leq \frac{1}{4}$ for all $s \geq s_1$, we have
\begin{equation}\label{equ:estWp1}
\int_{\mathbb{R}^n} |w|^{p+1}\rho dy \leq \frac{4(p+1)}{p-1}\mathcal{E}[w](s) + \frac{1}{2}\int_{\mathbb{R}^n} |w_s|^2\rho dy  + C, \quad \forall s > s_1.
\end{equation}
Substituting \eqref{equ:estWp1} into \eqref{equ:estDE3} yields \eqref{equ:estimateDE} with  $\tilde{s}_0 = \max\{\bar{s}_0,s_1\}$. This concludes the proof of Lemma \ref{lemm:lya} and Theorem \ref{theo:lya} also.
\end{proof}

\section{Blow-up behavior} \label{sec:refasy}
This section is devoted to the proof of Theorem \ref{theo:refinedasymptotic} and Theorem \ref{theo:pro}. Consider $b$ a blow-up point and write $w$ instead of $w_{b,T}$ for simplicity. From $(ii)$ of Theorem \ref{theo:blrate} and up to changing the signs of $w$ and $h$, we may assume that 
$\|w(y,s) -\kappa\|_{L^2_\rho} \to 0$ as $s \to +\infty$, uniformly on compact subsets of $\mathbb{R}^n$. As mentioned in the introduction, by setting $v(y,s) = w(y,s) - \phi(s)$ ($\phi$ is the positive solution of \eqref{equ:phiODE} such that $\phi(s) \to \kappa$ as $s \to + \infty$), we see that $\|v(y,s)\|_{L^2_\rho} \to 0$ as $s \to + \infty$ and $v$ solves the following equation: 
\begin{equation}\label{equ:v}
\partial_s v = (\mathcal{L} + \omega(s))v + F(v) + G(v,s), \quad \forall y \in \mathbb{R}^n,\; \forall s \in [-\log T, +\infty),
\end{equation}
where $\mathcal{L} = \Delta - \frac{y}{2}\cdot \nabla + 1$ and $\omega$, $F$, $G$ are given by
\begin{align*}
&\omega(s) = p\left(\phi^{p-1} - \kappa^{p-1}\right)+ e^{-s}h'\left(e^\frac{s}{p-1}\phi\right),\\
&F(v) = |v+ \phi |^{p-1}(v+\phi) - \phi^p - p\phi^{p-1}v, \\
&G(v,s) = e^{-\frac{ps}{p-1}}\left[h\left(e^\frac{s}{p-1}(v+\phi)\right)-h\left(e^\frac{s}{p-1}\phi\right) - e^\frac{s}{p-1}h'\left(e^\frac{s}{p-1}\phi\right)v\right].
\end{align*}
By a direct calculation, we can show that
\begin{equation}
|\omega(s)| = \mathcal{O}(\frac{1}{s^{a +1}}), \quad \text{as} \;\; s \to +\infty,
\end{equation}
(see Lemma \ref{lemm:apB1} for the proof of this fact, note also that in the case where $h$ is given by \eqref{equ:h1} and treated in \cite{NG14the}, we just obtain $|\omega(s)| = \mathcal{O}(s^{-a})$ as $s \to +\infty$, and that was a major reason preventing us from deriving the result in the case $a \in (0,1]$) in \cite{NG14the}.\\
Now introducing
\begin{equation}\label{def:V}
V(y,s) = \beta(s)v(y,s), \quad \text{where} \;\; \beta(s) = exp\left(-\int_s^{+\infty} \omega(\tau) d\tau\right), 
\end{equation}
then $V$ satisfies
\begin{equation}\label{equ:V}
\partial_s V = \mathcal{L}V + \bar{F}(V,s),
\end{equation}
where $\bar{F}(V,s) = \beta(s)(F(V) + G(V,s))$ satisfying 
\begin{equation}\label{equ:estFbar2}
\left|\bar{F}(V,s) - \frac{p}{2\kappa}V^2\right| = \mathcal{O}\left(\frac{V^2}{s^a}\right) + \mathcal{O}(|V|^3), \quad \text{as} \;\; s \to +\infty.
\end{equation}
(see Lemma C.1 in \cite{NG14the} for the proof of this fact, note that in the case where $h$ is given by \eqref{equ:h1}, the first term in the right-hand side of \eqref{equ:estFbar2} is $\mathcal{O}\left(\frac{V^2}{s^{a-1}}\right)$).\\

\noindent Since $\beta(s) \to 1$ as $s \to +\infty$, each equivalent for $V$ is also an equivalent for $v$. Therefore, it suffices to study the asymptotic behavior of $V$ as $s \to +\infty$. More precisely, we claim the following:
\begin{prop}[\textbf{Classification of the behavior of $V$ as $s \to +\infty$}]\label{prop:refinedasymptotic} One of the following possibilities occurs:\\
$i)\;$ $V(y,s) \equiv 0$,\\
$ii)$ There exists $l \in \{1, \dots, n\}$ such that up to an orthogonal transformation of coordinates, we have
$$
V(y,s) = -\frac{\kappa}{4ps} \left(\sum_{j=1}^l y_j^2 - 2l\right) + \mathcal{O}\left(\frac{1}{s^{a+1}} \right) + \mathcal{O}\left(\frac{\log s}{s^2}\right) \quad \text{as} \quad s \to +\infty.
$$
$iii)$ There exist an integer number $m \geq 3$ and constants $c_\alpha$ not all zero such that 
$$V(y,s) = - e^{\left(1 - \frac{m}{2}\right)s} \sum_{|\alpha| = m}c_\alpha H_\alpha(y) + o\left(e^{\left(1 - \frac{m}{2}\right)s}\right)\quad \text{as} \quad s \to +\infty.$$
The convergence takes place in $L^2_\rho$ as well as in $\mathcal{C}^{k,\gamma}_{loc}$ for any $k \geq 1$ and $\gamma \in (0,1)$.
\end{prop}
\begin{proof} Because we have the same equation \eqref{equ:V} and a similar estimate \eqref{equ:estFbar2} to the case treated in  \cite{NG14the}, we do not give the proof and kindly refer the reader to Section 3 in \cite{NG14the}.
\end{proof}

\noindent Let us derive Theorem \ref{theo:refinedasymptotic} from Proposition \ref{prop:refinedasymptotic}.
\begin{proof}[\textbf{Proof of Theorem \ref{theo:refinedasymptotic}}] By the definition \eqref{def:V} of $V$, we see that $i)$ of Proposition \ref{prop:refinedasymptotic} directly follows that $v(y,s) \equiv \phi(s)$ which is $i)$ of Theorem \ref{theo:refinedasymptotic}. Using $ii)$ of Proposition \ref{prop:refinedasymptotic} and the fact that $\beta(s) = 1 + \mathcal{O}(\frac{1}{s^{a}})$ as $s \to +\infty$, we see that as $s \to +\infty$,
\begin{align*}
w(y,s) &= \phi(s) + V(y,s)\left(1 + \mathcal{O}(\frac{1}{s^{a}})\right)\\
&= \phi(s) -\frac{\kappa}{4ps} \left(\sum_{j=1}^l y_j^2 - 2l\right) + \mathcal{O}\left(\frac{1}{s^{a+1}} \right) + \mathcal{O}\left(\frac{\log s}{s^2}\right),
\end{align*}
which yields $ii)$ of Theorem \ref{theo:refinedasymptotic}.\\
Using $iii)$ of Proposition \ref{prop:refinedasymptotic} and again the fact that $\beta(s) = 1 + \mathcal{O}(\frac{1}{s^{a}})$ as $s \to +\infty$, we have 
$$w(y,s) = \phi(s)- e^{\left(1 - \frac{m}{2}\right)s} \sum_{|\alpha| = m}c_\alpha H_\alpha(y) + o\left(e^{\left(1 - \frac{m}{2}\right)s}\right)\quad \text{as} \quad s \to +\infty.$$
This concludes the proof of Theorem \ref{theo:refinedasymptotic}.
\end{proof}

We now give the proof of Theorem \ref{theo:pro} from Theorem \ref{theo:refinedasymptotic}. Note that the derivation of Theorem \ref{theo:pro} from Theorem \ref{theo:refinedasymptotic} in the unperturbed case ($h \equiv 0$) was done by Vel\'azquez in \cite{VELcpde92}. The idea to extend the  convergence up to sets of the type $\{|y| \leq K_0 \sqrt{s}\}$ or $\{|y| \leq K_0e^{\left(\frac{1}{2} - \frac{1}{m}\right)s}\}$ is to estimate the effect of the convective term $-\frac{y}{2}\cdot\nabla w$ in the equation \eqref{equ:divw1} in $L^q_\rho$ spaces with $q > 1$. Since the proof of Theorem \ref{theo:pro} is actually in spirit by the method given in \cite{VELcpde92}, all that we need to do is to control the strong perturbation term in equation \eqref{equ:divw1}. We therefore give the main steps of the proof and focus only on the new arguments. Note also that we only give the proof of $ii)$ of Theorem \ref{theo:refinedasymptotic} because the proof of $iii)$ is exactly the same as written in Proposition 34 in \cite{NG14the}.\\

Let us restate $i)$ of Theorem \ref{theo:pro} in the following proposition:
\begin{prop}[\textbf{Asymptotic behavior in the $\frac{y}{\sqrt{s}}$ variable}] \label{prop:1} Assume that $w$ is a solution of equation \eqref{equ:divw1} which satisfies $ii)$ of Theorem \ref{theo:refinedasymptotic}. Then, for all $K > 0$, 
$$\sup_{|\xi| \leq K} \left|w(\xi\sqrt{s},s) - f_l(\xi)\right| = \mathcal{O}\left(\frac{1}{s^{a}} \right) + \mathcal{O}\left(\frac{\log s}{s} \right), \quad \text{as} \quad s \to + \infty,$$
where $f_l(\xi) = \kappa\left(1 + \frac{p-1}{4p}\sum_{j = 1}^l\xi_j^2\right)^{-\frac{1}{p-1}}$. 
\end{prop}
\begin{proof} Define $q = w - \varphi$, where
\begin{equation}\label{equ:defiv}
\varphi(y,s) = \frac{\phi(s)}{\kappa}\left[ \kappa\left(1+ \frac{p-1}{4ps}\sum_{j=1}^l y_j^2\right)^{-\frac{1}{p-1}} + \frac{\kappa l}{2ps}\right],
\end{equation}
and $\phi$ is the unique positive solution of \eqref{equ:phiODE} satisfying \eqref{equ:solphi}.\\
Note that in \cite{VELcpde92} and \cite{NG14the}, the authors took $\varphi(y,s) =  \kappa\left(1+ \frac{p-1}{4ps}\sum_{j=1}^l y_j^2\right)^{-\frac{1}{p-1}} + \frac{\kappa l}{2ps}$. But this choice just works in the case where $a > 1$. In the particular case \eqref{equ:h}, we use in additional the factor $\frac{\phi(s)}{\kappa}$ which allows us to go beyond the order $\frac{1}{s^a}$ coming from the strong perturbation term in order to reach $\frac{1}{s^{a+1}}$ in many estimates in the proof.\\

\noindent Using Taylor's formula in \eqref{equ:defiv} and $ii)$ of Theorem \ref{theo:refinedasymptotic}, we find that 
\begin{equation}
\|q(y,s)\|_{L^2_\rho} = \mathcal{O}\left(\frac{1}{s^{a+1}} \right) + \mathcal{O}\left(\frac{\log s}{s^2} \right), \quad \text{as} \quad s \to + \infty.
\end{equation}
Straightforward calculations based on equation \eqref{equ:divw1} yield 
\begin{equation}\label{eq:Wa}
\partial_s q = (\mathcal{L} + \alpha)q + F(q) + G(q,s) + R(y,s), \quad \forall (y,s) \in \mathbb{R}^n \times [-\log T, +\infty), 
\end{equation}
where 
\begin{align*}
\alpha(y,s) & = p(\varphi^{p-1} - \kappa^{p-1}) + e^{-s}h'\left(e^\frac{s}{p-1}\varphi\right),\\
F(q)& = |q + \varphi|^{p-1}(q + \varphi) - \varphi^p -p\varphi^{p-1}q,\\
G(q,s)& = e^{-\frac{ps}{p-1}}\left[h\left(e^\frac{s}{p-1}(q + \varphi)\right) - h\left(e^\frac{s}{p-1}\varphi\right) - e^\frac{s}{p-1}h'\left(e^\frac{s}{p-1}\varphi\right)q\right],\\
R(y,s)&= -\partial_s \varphi + \Delta \varphi - \frac{y}{2}\cdot \nabla \varphi - \frac{\varphi}{p-1} + \varphi^p + e^{-\frac{ps}{p-1}}h\left(e^\frac{s}{p-1}\varphi\right).
\end{align*}
Let $K_0 > 0$ be fixed, we consider first the case $|y| \geq 2K_0\sqrt{s}$ and then $|y| \leq 2K_0\sqrt{s}$ and make a Taylor expansion for $\xi = \frac{y}{\sqrt{s}}$ bounded. Simultaneously, we obtain for all $s \geq s_0$, 
$$\alpha(y,s) \leq \frac{C_1}{s^{a'}},$$
$$|F(q)| + |G(q,s)| \leq C_1(q^2+ \mathbf{1}_{\{|y| \geq 2K_0\sqrt{s}\}}),$$
$$|R(y,s)| \leq C_1\left( \frac{|y|^2 + 1}{s^{1+a'}} + \mathbf{1}_{\{|y| \geq 2K_0\sqrt{s}\}}\right),$$
where $a' = \min\{1,a\}$, $C_1 = C_1(M_0, K_0) > 0$, $M_0$ is the bound of $w$ in $L^\infty$-norm. Note that we need to use in addition the fact that $\phi$ satisfies equation \eqref{equ:phiODE} to derive the bound for $R$ (see Lemma \ref{lemm:apB2}).\\
Let $Q = |q|$, we then use the above estimates and Kato's inequality, i.e  $\Delta f \cdot \text{sign}(f) \leq \Delta(|f|)$, to derive from equation \eqref{eq:Wa} the following: for all $K_0 > 0$ fixed, there are $C_* = C_*(K_0,M_0) > 0$ and a time $s' > 0$ large enough such that for all $s \geq s_* = \max\{s',-\log T\}$,
\begin{equation}\label{eq:Qa}
\partial_s Q \leq \left(\mathcal{L} + \frac{C_*}{s^{a'}} \right)Q + C_*\left(Q^2 + \frac{(|y|^2+1)}{s^{1 + a'}} + \mathbf{1}_{\{|y| \geq 2K_0\sqrt{s}\}} \right), \quad \forall y \in \mathbb{R}^n.
\end{equation}
Since 
$$\left|w(y,s) - f_l\left(\frac{y}{\sqrt{s}}\right) \right| \leq Q + \frac{C}{s^{a'}},$$
the conclusion of  Proposition \ref{prop:1} follows if we show that 
\begin{equation}\label{equ:pr1}
\forall K_0 > 0,\quad \sup_{|y| \leq K_0\sqrt{s}} Q(y,s) \to 0 \quad \text{as} \quad s \to + \infty.
\end{equation}
Let us now focus on the proof of \eqref{equ:pr1} in order to conclude Proposition \ref{prop:1}. For this purpose, we introduce the following norm: for $r \geq 0$, $q > 1$ and $f \in L^q_{loc}(\mathbb{R}^n)$, 
$$L_\rho^{q,r}(f) \equiv \sup_{|\xi| \leq r }\left(\int_{\mathbb{R}^n}|f(y)|^q\rho(y - \xi)dy \right)^\frac{1}{q}.$$
Following the idea in \cite{VELcpde92}, we shall make estimates on solution of \eqref{eq:Qa} in the $L^{2,r(\tau)}_\rho$ norm where $r(\tau) = K_0e^{\frac{\tau - \bar{s}}{2}} \leq K_0\sqrt{\tau}$. Particularly, we have the following:
\begin{lemm}\label{lemm:gt} Let $s$ be large enough and $\bar{s}$ is defined by $e^{s - \bar{s}} = s$. Then for all $\tau \in [\bar{s}, s]$ and for all $K_0 > 0$, it holds that 
$$ g(\tau) \leq 
C_0\left( e^{\tau - \bar{s}}\epsilon(\bar{s}) + \int_{\bar{s}}^{(\tau - 2K_0)_+}\frac{e^{(\tau - t - 2K_0)}g^2(t)}{\left(1 - e^{-(\tau - t - 2K_0)} \right)^{1/20}}dt \right)$$
where $g(\tau) = L^{2,r(K_0,\tau,\bar{s})}_\rho (Q(\tau))$, $r(K_0, \tau, \bar{s}) = K_0e^{\frac{\tau - \bar{s}}{2}}$, $\epsilon(s) = \mathcal{O}\left(\frac{1}{s^{a+1}} \right) + \mathcal{O}\left(\frac{\log s}{s^2} \right)$, $C_0 = C_0(C_*, M_0, K_0)$ and $z_+ = \max\{z,0\}$.
\end{lemm}
\begin{proof} Multiplying \eqref{eq:Qa} by $\beta(\tau) = e^{\int_{\bar{s}}^\tau \frac{C_*}{t^{a'}}}dt$, then we write $Q(y,\tau)$ for all $(y,\tau) \in \mathbb{R}^n \times [\bar{s},s]$ in the integration form:
\begin{align*}
Q(y,\tau) &= \beta(\tau)S_\mathcal{L}(\tau - \bar{s})Q(y,\bar{s})\\
& + C_*\int_{\bar{s}}^\tau \beta(\tau)S_\mathcal{L}(\tau - t)\left(Q^2 + \frac{|y|^2}{t^{1 + a'}} + \frac{1}{t^{1 + a'}} + \mathbf{1}_{\{|y| \geq 2K_0\sqrt{t}\}} \right)dt,
\end{align*}
where $S_\mathcal{L}$ is the linear semigroup corresponding to the operator $\mathcal{L}$.\\
Next, we take the $L^{2, r(K_0, \tau, \bar{s})}_\rho$-norms both sides in order to get the following:
\begin{align*}
g(\tau) &\leq C_0L^{2,r}_\rho \big[S_\mathcal{L}(\tau -\bar{s})Q(\bar{s})\big] + C_0\int_{\bar{s}}^\tau L^{2,r}_\rho \big[S_\mathcal{L}(\tau -t)Q^2(t)\big]dt \\
&+ C_0\int_{\bar{s}}^\tau L^{2,r}_\rho \left[S_\mathcal{L}(\tau -t)\left(\frac{|y|^2}{t^{1 + a'}} + \frac{1}{t^{1 + a'}}\right) \right]dt\\
&+ C_0\int_{\bar{s}}^\tau L^{2,r}_\rho \big[S_\mathcal{L}(\tau -t)\mathbf{1}_{\{|y| \geq 2K_0\sqrt{t}\}}\big]dt \equiv J_1 + J_2 + J_3 + J_4.
\end{align*}
Proposition 2.3 in \cite{VELcpde92} yields
$$|J_1| \leq C_0e^{\tau - \bar{s}}\|Q(\bar{s})\|_{L^2_\rho} = e^{\tau - \bar{s}}\mathcal{O}(\epsilon(\bar{s})) \quad \text{as} \quad \bar{s} \to + \infty,$$
$$|J_2| \leq \frac{C_0}{\bar{s}^{1 + a'}}e^{\tau - \bar{s}} + C_0\int_{\bar{s}}^{(\tau - 2K_0)_+}\frac{e^{(\tau - t - 2K_0)}}{\left(1 - e^{-(\tau - t - 2K_0)} \right)^{1/20}}\left[L_\rho^{2,r(K_0, t, \bar{s})}Q(t)\right]^2dt,$$
$$|J_3| \leq\frac{C_0e^{\tau - \bar{s}}}{\bar{s}^{1 + a'}}(1 + (\tau - \bar{s})),$$
$$|J_4| \leq C_0e^{-\delta \bar{s}}, \quad \text{where} \quad \delta = \delta(K_0) > 0.$$
Putting together the estimates on $J_i, i =1, 2, 3, 4$, we conclude the proof of Lemma \ref{lemm:gt}.
\end{proof}
\noindent We now use the following Gronwall lemma from Vel\'azquez \cite{VELcpde92}:
\begin{lemm}[\textbf{Vel\'azquez \cite{VELcpde92}}]\label{lem:Gro} Let $\epsilon, C, R$ and $\delta$ be positive constants, $\delta \in (0,1)$. Assume that $H(\tau)$ is a family  of continuous functions satisfying
$$\mathcal{H}(\tau) \leq \epsilon e^\tau + C\int_0^{(\tau - R)_+} \frac{e^{\tau -s} \mathcal{H}^2(s)}{\left(1 - e^{-(\tau - s - R)} \right)^\delta}ds, \quad \text{for $\tau > 0$}.$$
Then there exist $\theta = \theta(\delta, C, R)$ and $\epsilon_0 = \epsilon_0(\delta, C, R)$ such that for all $\epsilon \in (0,\epsilon_0)$ and any $\tau$ for which $\epsilon e^\tau \leq \theta$, we have
$$\mathcal{H}(\tau) \leq 2\epsilon e^\tau.$$
\end{lemm}

\noindent Applying Lemma \ref{lem:Gro} with $\mathcal{H} \equiv g$, we see from Lemma \ref{lemm:gt} that for $s$ large enough,
$$g(\tau) \leq 2C_0e^{\tau - \bar{s}}\epsilon(\bar{s}), \quad \forall \tau \in [\bar{s},s].$$
If $\tau = s$, then $e^{s - \bar{s}} = s$, $r = K_0\sqrt{s}$ and 
$$g(s) \equiv L^{2, K_0\sqrt{s}}_\rho\big(Q(s)\big) = \mathcal{O}\left(\frac{1}{s^{a}} \right) + \mathcal{O}\left(\frac{\log s}{s} \right),\; \text{as} \quad s \to + \infty.$$
By using the regularizing effects of the semigroup $S_{\mathcal{L}}$ (see Proposition 2.3 in \cite{VELcpde92}), we then obtain
$$\sup_{|y| \leq \frac{K_0\sqrt{s}}{2}} Q(y,s) \leq C'(C_*,K_0, M_0)L^{2,K_0\sqrt{s}}_\rho(Q(s)) =  \mathcal{O}\left(\frac{1}{s^{a}} \right) + \mathcal{O}\left(\frac{\log s}{s} \right),$$
as $s \to +\infty$, which concludes the proof of Proposition \ref{prop:1}.
\end{proof}

\section{Numerical method}\label{sec:numme}
We give in this section a numerical study of the blow-up profile of equation \eqref{equ:problem} in one dimension. Though our method is very similar  to Berger and Kohn's algorithm \cite{BKcpam88} in spirit, it is better in the sense that is can be applied to equations which are not invariant under the transformation \eqref{equ:invE}. Our method differs from Berger and Kohn's in the following way: we step the solution forward until its maximum value multiplied by a power of its mesh size reaches a preset threshold, where the mesh size and the preset threshold are linked;  for the rescaling algorithm, the solution is stepped forward until its maximum value reaches a preset threshold, and the mesh size and the preset threshold do not need to be linked. For more clarity, we present in the next subsection the mesh-refinement technique applied to equation \eqref{equ:problem}, then give various numerical experiments to illustrate the effectiveness of our method for the problem of the numerical blow-up profile. Note that our method is more general than Berger and Kohn's \cite{BKcpam88}, in the sense that it applies to non scale invariant equations. However, when applied to the unperturbed case $F(u) = |u|^{p-1}u$, our method gives exactly the same approximation as that of \cite{BKcpam88}.

\subsection{Mesh-refinement algorithm}\label{sec:subag}
In this section, we describe our refinement algorithm to solve numerically the problem \eqref{equ:problem} with initial data $\varphi(x) > 0$, $\varphi(x) = \varphi(-x)$, $x\frac{d \varphi(x)}{dx} < 0$ for  $x \ne 0$, which gives a positive symmetric and radially decreasing solution. Let us rewrite the problem \eqref{equ:problem} (with $\mu = 1$) in the following: 
\begin{equation}\label{equ:log}
\left\{
\begin{array}{lrl}
u_t &=& u_{xx} + F(u), \quad (x,t) \in (-1,1)\times (0,T),\\
u(1,t)&=& u(-1,t) = 0, \quad t \in (0,T),\\
u(x,0)&=& \varphi(x), \quad x \in (-1,1),
\end{array}
\right.
\end{equation}
where $p > 1$ and 
\begin{equation}\label{def:nolF}
F(u) = u^p + \frac{u^p}{\log^{a}(2 + u^2)}\quad \text{with}\;\; a > 0.
\end{equation}

\noindent Let $\hbar$ and $\tau$ be the initial space and time steps, we define $C_\Delta = \frac{\tau}{\hbar^2}$, $x^i = i\hbar$, $t^n = n\tau$, $I = \frac 1 \hbar$ and $u^{i,n}$ as the approximation of $u(x^i,t^n)$, where $u^{i,n}$ is defined for all  $n \geq 0$, for all $i \in \{-I, \dots, I\}$ by
\begin{align}
&u^{i,n+1} = u^{i,n} + C_\Delta\left[u^{i-1,n} - 2 u^{i,n} + u^{i+1,n}\right] + \tau F\left(u^{i,n}\right),\label{equ:scheme}\\
&u^{I,n} = u^{-I,n} = 0, \quad u^{i,0} = \varphi_i.\nonumber
\end{align}
Note that this scheme is first order accurate in time and second order in space, and it requests the stability condition $C_\Delta = \frac{\tau}{\hbar^2} \leq \frac{1}{2}$.\\
\noindent Our algorithm needs to fix the following parameters: 
\begin{itemize}
\item $\lambda < 1$: the refining factor with $\lambda^{-1}$ being a small integer. 
\item $M$: the threshold to control the amplitude of the solution,
\item $\alpha$: the parameter controlling the width of interval to be refined.
\end{itemize}
The parameters $\lambda$ and $M$ must satisfy the following relation:
\begin{equation}\label{rel:Mlambda}
M = \lambda^{-\frac{2}{p-1}}M_0, \quad \text{where} \quad M_0 = \hbar^\frac{2}{p-1}\|\varphi\|_\infty.
\end{equation}
Note that the relation \eqref{rel:Mlambda} is important to make our method works. In \cite{BKcpam88}, the typical choice is $M_0 = \|\varphi\|_\infty$, hence $M = \lambda^{-\frac{2}{p-1}}\|\varphi\|_\infty$.\\

\noindent In the initial step of the algorithm, we simply apply the scheme \eqref{equ:scheme} until $\hbar^\frac{2}{p-1} \|u(\centerdot, t^n)\|_\infty$ reaches $M$ (note that in \cite{BKcpam88} the solution is stepped forward until $\|u(\centerdot, t^n)\|_\infty$ reaches $M$; in this first step, the thresholds of the two methods are the same, however, they will split after the second step; roughly speaking, for the threshold we shall use the quantity $\hbar^\frac{2}{p-1}\|u(\centerdot, t^n)\|_\infty$ in our method instead of $\|u(\centerdot, t^n)\|_\infty$ in \cite{BKcpam88}). Then, we use a linear interpolation in time to find $\tau_0^*$ such that
$$t^n - \tau \leq \tau_0^* \leq t^n \quad \text{and} \quad \hbar^\frac{2}{p-1} \|u(\centerdot, \tau_0^*)\| = M.$$
Afterward, we determine two grid points $y_0^-$ and $y_0^+$ such that 
\begin{equation}\label{equ:defiy_0}
\left\{
\begin{aligned}
&\hbar^\frac{2}{p-1} u(y_0^- - \hbar, \tau_0^*) < \alpha M \leq \hbar^\frac{2}{p-1} u(y_0^-, \tau_0^*) \\
&\hbar^\frac{2}{p-1} u(y_0^+ + \hbar, \tau_0^*) < \alpha M \leq \hbar^\frac{2}{p-1} u(y_0^+, \tau_0^*).
\end{aligned}
\right.
\end{equation}
Note that $y_0^- = - y_0^+$ because of the symmetry of the solution. This closes the initial step.

\noindent Let us begin the first refining step. Define
\begin{equation}\label{equ:rel1}
u_1(y_1, t_1) = u(y_1, \tau_0^* + t_1), \quad y_1 \in (y_0^-, y_0^+),\; t_1 \geq 0,
\end{equation}
and setting $h_1 = \lambda \hbar$, $\tau_1 = \lambda^2\tau$ as the space and time step for the approximation of $u_1$ (note that $\frac{\tau_1}{h_1^2} = \frac{\tau}{\hbar^2} = C_\Delta$ which is a constant), $y_1^i = ih_1$, $t_1^n = n\tau_1$, $I_1 = \frac{y_0^+}{h_1}$ and $u_1^{i,n}$ as the approximation of $u_1(y_1^i,t_1^n)$ (note that in the unperturbed case, Berger and Kohn used the transformation \eqref{equ:invE} to define      $u_1(y_1, t_1) = \lambda^{\frac{2}{p-1}}u(\lambda y_1, \tau_0^* + \lambda^2t_1)$, and then applied the same scheme for $u$ to $u_1$. However, we can not do the same because the equation \eqref{equ:log} is not in fact invariant under the transformation \eqref{equ:invE}). Then applying the scheme \eqref{equ:scheme} to $u_1$ which reads 
\begin{align}\label{equ:scheme1}
u_1^{i,n+1} = u_1^{i,n} + C_\Delta\left[u_1^{i-1,n} - 2u_1^{i,n} + u_1^{i+1,n}\right] + \tau_1 F\left(u_1^{i,n}\right),
\end{align}
for all  $n \geq 0$ and for all  $i \in \{-I_1+1, \dots, I_1-1\}$.\\
Note that the computation of $u_1$ requires the initial data $u_1(y_1,0)$ and the boundary condition $u_1(y_0^\pm,t_1)$. For the initial condition, it is determined from $u(x,\tau_0^*)$ by using interpolation in space to get values at the new grid points. For the boundary condition, since $\tau_1 = \lambda^2 \tau$, we then have from \eqref{equ:rel1},
\begin{equation}\label{equ:rel2}
u_1(y_0^\pm, n\tau_1) = u(y_0^\pm, \tau_0^* + n\lambda^2\tau).
\end{equation}
Since $u$ and $u_1$ will be stepped forward, each on its own grid ($u_1$ on $(y_0⁻, y_0^+)$ with the space and time step $h_1$ and $\tau_1$, and $u$ on $(-1, 1)$ with the space and time step $\hbar$ and $\tau$), the relation  \eqref{equ:rel2} will provide us with the boundary values for $u_1$. In order to better understand how it works, let us consider an example with $\lambda = \frac{1}{2}$. After closing the initial phase, the two solutions $u_1$ and $u$ are stepped forward independently, each on its own grid, in other words, $u_1$ on $(y_0⁻, y_0^+)$ with the space and time step $h_1$ and $\tau_1$, and $u$ on $(-1, 1)$ with the space and time step $\hbar$ and $\tau$. Then using the linear interpolation in time for $u$, we get the boundary values for $u_1$ by \eqref{equ:rel2}. Since $\tau_1 = \lambda^2 \tau = \frac{1}{4}\tau$. This means that $u$ is stepped forward once every 4 time steps of $u_1$. After 4 steps forward of $u_1$, the values of $u$ on the interval $(y_0^- , y_0^+)$ must be updated to agree with the calculations of $u_1$. In other words, the approximation of $u$ is used to assist in computing the boundary values for $u_1$. At each successive time step for $u$, the values of $u$ on the interval $(y_0^-, y_0^+)$ must be updated to make them agree with the more accurate fine grid solution $u_1$. When $h_1^\frac{2}{p-1}\|u_1(\centerdot, n\tau_1)\|_\infty$ first exceeds $M$, we use a linear interpolation in time to find $\tau_1^* \in [\tau_1^{n-1}, \tau_1^n]$ such that $h_1^\frac{2}{p-1} \|u_1(\centerdot, \tau_1^*)\|_\infty = M$. On the interval where $h_1^\frac{2}{p-1}\|u_1(\centerdot, \tau_1^*)\|_\infty > \alpha M$, the grid is refined further and the entire procedure as for $u_1$ is repeated to yield $u_2$ and so forth.\\

\noindent Before going to a general step, we would like to comment on the relation \eqref{rel:Mlambda}. Indeed, when $\hbar^\frac{2}{p-1} \| u(\centerdot, t)\|_\infty$ reaches the given threshold $M$ in the initial phase, namely when $\hbar^\frac{2}{p-1} \| u(\centerdot, \tau_0^*)\|_\infty = M$, we want to refine the grid such that the maximum values of $h_1^\frac{2}{p-1}u_1(y_1,0)$ equals to $M_0$. By \eqref{equ:rel1}, this request turns into $h_{1}^\frac{2}{p-1} \| u(\centerdot, \tau_0^*)\|_\infty = M_0$. Since $h_1 = \lambda \hbar$, it follows that $M = \lambda^{-\frac{2}{p-1}}M_0$, which yields \eqref{rel:Mlambda}.\\

\noindent Let $k \geq 0$, we set $h_{k+1} = \lambda^{-1}h_{k}$ and $\tau_{k+1} = \lambda^2\tau_k$ (note that $\frac{\tau_{k+1}}{h_{k+1}^2} = \frac{\tau_{k}}{h_{k}^2} = C_\Delta$ which is a constant), $y_{k+1}$ and $t_{k+1}$ as the variables of $u_{k+1}$, $y_k^i = ih_k$, $t_k^n = n \tau_k$. The index $k = 0$ means that $u_0(y_0,t_0) \equiv u(x,t)$, $h_0 \equiv \hbar$ and $\tau_0 \equiv \tau$. The solution $u_{k+1}$ is related to $u_k$ by 
\begin{equation}\label{equ:relukuk1}
u_{k+1}(y_{k+1}, t_{k+1}) = u_k(y_{k+1}, \tau_k^* + t_{k+1}), \quad y_{k+1} \in (y_k^-, y_k^+), \; t_{k+1} \geq 0.
\end{equation}
Here, the time $\tau_k^* \in [t_k^{n-1}, t_k^n]$ satisfies $h_k^{\frac{2}{p-1}}\|u_k(\cdot, \tau_k^*)\|_{\infty} = M$, and $y_k^-,y_k^+$ are two grid points determined by
\begin{equation}\label{equ:defiy_k}
\left\{
\begin{aligned}
& h_k^\frac{2}{p-1}u_k(y_k^- - h_k, \tau_k^*) < \alpha M \leq h_k^\frac{2}{p-1}u_k(y_k^-, \tau_k^*),\\
& h_k^\frac{2}{p-1}u_k(y_k^+ + h_k, \tau_k^*) < \alpha M \leq h_k^\frac{2}{p-1}u_k(y_k^+, \tau_k^*).
\end{aligned}
\right.
\end{equation}
The approximation of $u_{k+1}(y_{k+1}^i, t_{k+1}^n)$ (denoted by $u_{k+1}^{i,n}$) uses the scheme \eqref{equ:scheme} with the space step $h_{k+1}$ and the time step $\tau_{k+1}$, which reads
\begin{align}
u_{k+1}^{i,n+1} &= u_{k+1}^{i,n} + C_{\Delta}\left[u_{k+1}^{i-1,n} - 2u_{k+1}^{i,n} + u_{k+1}^{i+1,n}\right] + \tau_{k+1} F\left(u_{k+1}^{i,n}\right),\label{equ:scheme2}
\end{align}
for all $n \geq 1$ and $i \in \{-I_k+1,\cdots, I_k - 1\}$ with $I_k = \frac{y_k^+}{h_{k+1}}$ (note from introduction that $I_k$ is an integer since $\lambda^{-1} \in \mathbb{N}$).\\
As for the approximation of $u_k$, the computation of $u_{k+1}^{i,n}$ needs the initial data and the boundary condition. From \eqref{equ:relukuk1} and the fact that $\tau_{k+1} = \lambda^2 \tau_k$, we see that 
\begin{equation} \label{equ:reltmpk1}
u_{k+1}(y_{k+1},0) = u_k(y_{k+1}, \tau_k^*) \quad \text{and} \quad u_{k+1}(y_k^\pm, n\tau_{k+1}) = u_k(y_k^\pm, \tau_k^* + n\lambda^2\tau_k).
\end{equation}
Hence, from the first identity in \eqref{equ:reltmpk1}, the initial data is simply calculated from $u_k(\cdot, \tau_k^*)$ by using a linear interpolation in space in order to assign values at new grid points. The essential step in this new mesh-refinement method is to determine the boundary condition through the second identity in \eqref{equ:reltmpk1}. This means by a linear interpolation in time of $u_k$. Therefore, the previous solutions $u_k$, $u_{k-1}$, $\cdots$ are stepped forward independently, each on its own grid. More precisely, since $\tau_{k+1} = \lambda^2\tau_k = \lambda^4\tau_{k-1} = \cdots$, then $u_k$ is stepped forward once every $\lambda^{-2}$ time steps of $u_{k+1}$; $u_{k-1}$ once every $\lambda^{-4}$ time steps of $u_{k+1}$, ... On the other hand, the values of $u_k$, $u_{k-1}$, ... must be updated to agree with the calculation of $u_{k+1}$.  When $h_{k+1}^\frac{2}{p-1}\|u_{k+1}(\centerdot, \tau_{k+1})\|_\infty > M$, then it is time for the next refining phase. \\

\noindent We would like to comment on the output of the refinement algorithm:
\begin{itemize}
\item[i)] Let $\tau_k^*$ be the time at which the refining takes place, then the ratio $\frac{\tau_k^*}{\tau_{k}}$, which indicates the number of time steps until $h_k^{\frac{2}{p-1}}\|u_k\|_\infty$, reaches the given threshold $M$, is independent of $k$ and tends to a constant as $k \to \infty$.
\item[ii)] Let $u_k(\cdot, \tau_k^*)$ be the \emph{refining solution}. If we plot $h_k^\frac{2}{p-1}u_k(\centerdot, \tau_k^*)$ on $(-1,1)$, then their graphs are eventually independent of $k$ and converge as $k \to \infty$.
\item[iii)] Let $(y_k^-,y_k^+)$ be the interval to be refined, then the quality $(h_k^{-1}y_k^+)^2$ behaves as a linear function of $k$. 
\end{itemize}  
These assertions can be well understood by the following theorem:
\begin{theo}[Formal analysis]\label{theo:3} Let $u$ be a blowing-up solution to equation \eqref{equ:log}, then the output of the refinement algorithm satisfies:\\
$i)$ The ratio $\frac{\tau_k^*}{\tau_{k}}$ is independent of $k$ and tends to a constant as $k \to \infty$, namely 
\begin{equation}\label{equ:Nasymp}
\frac{\tau_k^*}{\tau_{k}} \to\frac{(\lambda^{-2}-1)M^{1-p}}{C_\Delta(p - 1)}, \quad \text{as} \;\; k \to +\infty. 
\end{equation}
Assume in addition that $i)$ of Theorem \ref{theo:pro} holds,\\
$ii)$ Defining $v_k(z) = h_k^\frac{2}{p-1}u_k(zy_{k-1}^+, \tau_k^*)$ for all $k \geq 1$, we have
\begin{equation}\label{equ:vk}
\forall |z| <  1,\;\; v_k(z) \sim M\left(1 + (\alpha^{1-p} - 1)\lambda^{-2}z^2\right)^{-\frac{1}{p-1}} \quad \text{as} \;\; k \to +\infty.
\end{equation}
$iii)$ The quality $(h_k^{-1}y_k^+)^2$ behaves as a linear function, namely
\begin{equation}
(h_k^{-1}y_k^+)^2  \sim \gamma k + B \quad \text{as} \;\; k \to +\infty.
\end{equation}
where $\gamma = \frac{2M^{1-p}(\alpha^{1-p} - 1)|\log \lambda|}{c_p(p-1)\lambda^2}$, $B = -\frac{M^{1-p}(\alpha^{1-p} - 1)}{c_p(p-1)\lambda^2}\log \left(\frac{M^{1-p}\hbar^2}{p-1}\right)$ and $c_p = \frac{p-1}{4p}$ .
\end{theo}
\begin{rema} Note that there is no assumption on the value of $a$ in the hypothesis in Theorem \ref{theo:3}. It is understood in the sense that $u$ blows up in finite time and its profile is described in Theorem \ref{theo:pro}.
\end{rema}
\begin{proof} As we will see in the proof that the statement $i)$ concerns the blow-up limit of the solution and the second one is due to the blow-up profile stated in Theorem \ref{theo:pro}.\\
$i)$ Let $\sigma_k$ is the real time when the refinement from $u_k$ to $u_{k+1}$ takes place, we have by \eqref{equ:relukuk1}, 
$$\sigma_k = \tau_0^* + \tau_1^* + \dots  + \tau_k^*,$$
where $\tau_j^*$ is such that $h_j^{\frac{2}{p-1}}\|u_k(\cdot, \tau_j^*)\|_\infty = M$. This means that
\begin{equation}\label{equ:reluk2u}
u_k(\cdot,\tau_k^*) = u(\cdot, \sigma_k).
\end{equation}
On the other hand, from $i)$ of Theorem \ref{theo:pro} and the definition \eqref{def:f} of $f$, we see that
\begin{equation}\label{equ:blrate}
\lim_{t \to T} (T-t)^{\frac{1}{p-1}}\|u(\centerdot, t)\|_{L^\infty} = \kappa.
\end{equation}
Combining \eqref{equ:blrate} and \eqref{equ:reluk2u} yields
\begin{equation}\label{equ:t_TkU}
(T - \sigma_k)^{\frac{1}{p-1}}\|u_k(\centerdot, \tau_k^*)\|_\infty = \kappa + o(1),
\end{equation}
where $o(1)$ represents a term that tends to $0$ as $k \to +\infty$. \\
Since $\|u_k(\centerdot, \tau_k^*)\|_\infty = Mh_k^{\frac{-2}{p-1}}$, we then derive 
\begin{equation}\label{equ:T_tk}
T - \sigma_k= \left(M^{-1}\kappa\right)^{p-1}h_k^2 + o(1).
\end{equation}
By the definition of $\sigma_k$ and \eqref{equ:reluk2u}, we infer that $\tau_k^* = \sigma_k - \sigma_{k-1}$ (we can think $\tau_k^*$ as the \emph{live time} of $u_k$ in the $k$-th refining phase). Hence,
\begin{align*}
\frac{\tau_k^*}{\tau_{k}} = \frac{\sigma_k - \sigma_{k-1}}{\tau_{k}} & = \frac{1}{\tau_{k}}\left[(T - \sigma_{k-1}) - (T - \sigma_k)\right]\\
& = \frac{1}{\tau_{k}}\left(M^{-1}\kappa\right)^{p-1}(h_{k-1}^2 - h_k^2) + o(1)\\
& = \frac{h_k^2}{\tau_{k}}\left(M^{-1}\kappa\right)^{p-1}(\lambda^{-2} - 1) + o(1).
\end{align*} 
Since the ratio $\frac{\tau_k}{h_k^2}$ is always fixed by the constant $C_\Delta$, we finally obtain
$$
\lim_{k \to +\infty}\frac{\tau_k^*}{\tau_{k}} = \frac{(\lambda^{-2}-1)M^{1-p}}{C_\Delta(p - 1)},
$$
which concludes the proof of part $i)$ of Theorem \ref{theo:3}.\\

\noindent $ii)$  Since the symmetry of the solution, we have $y^-_{k-1} = y^+_{k-1}$. We then consider the following mapping: for all $k \geq 1$,
\begin{equation*}
\forall |z| \leq 1,\;\; z \mapsto v_k(z), \quad \text{where}\quad v_k(z) = h_k^\frac{2}{p-1}u_k(zy_{k-1}^+, \tau_k^*).
\end{equation*}
We will show that $v_k(z)$ is independently of $k$ and converges as $k \to +\infty$. For this purpose, we first write $u_k(y_k, \tau_k*)$ in term of $w(\xi, s)$ thanks to \eqref{equ:reluk2u} and \eqref{equ:simivariables}, 
\begin{equation}\label{equ:tmpgjjfsj}
u_k(y_k, \tau_k^*) = u(y_k, \sigma_k) = (T- \sigma_k)^{-\frac{1}{p-1}}w(\xi_k, s_k),
\end{equation}
where $\xi_k = \frac{y_k}{\sqrt{T - \sigma_k}}$ and $s_k = -\log(T-\sigma_k)$.\\
If we write $i)$ of Theorem \ref{theo:pro} in the variable $\frac{y}{\sqrt{s}}$ through \eqref{equ:simivariables}, we have the following equivalence:
\begin{equation}\label{equ:limitw2f}
\left\|w(y,s) - f\left(\frac{y}{\sqrt{s}}\right) \right\|_{L^\infty} \to 0 \quad \text{as $s \to +\infty$},
\end{equation}
where $f$ is given in \eqref{def:f}.\\
From \eqref{equ:limitw2f}, \eqref{equ:T_tk} and \eqref{equ:tmpgjjfsj}, we derive 
$$u_k(y_k,\tau_k^*) = M \kappa^{-1}h_k^{-\frac{2}{p-1}} f \left(\frac{y_k}{(M^{-1}\kappa)^{\frac{p-1}{2}} h_k \sqrt{s_k}}\right) + o(1).$$
Then multiplying both of sides by $h_k^{\frac{2}{p-1}}$ and replacing $y_k$ by $zy_{k-1}^+$, we obtain
\begin{equation}\label{equ:rel_uk_W}
h_k^{\frac{2}{p-1}}u_k(zy_{k-1}^+, \tau_k^*)  = M \kappa^{-1} f \left(\frac{z y_{k-1}^+}{(M^{-1}\kappa)^{\frac{p-1}{2}} h_k \sqrt{s_k}}\right) + o(1).
\end{equation}
From the definition \eqref{equ:defiy_k} of $y_{k-1}^+$, we may assume that
$$h_{k-1}^{\frac{2}{p-1}}u_{k-1}(y_{k-1}^+, \tau_{k-1}^*) = \alpha M.$$
Combining this with \eqref{equ:rel_uk_W}, we have
$$\alpha = \kappa^{-1} f \left(\frac{y_{k-1}^+}{(M^{-1}\kappa)^{\frac{p-1}{2}} h_{k-1} \sqrt{s_{k-1}}}\right) + o(1).$$
Since $s_k = -\log(T - \sigma_k)$ and the fact that $h_k = \lambda^k \hbar$,  we have from \eqref{equ:T_tk} that
\begin{equation}\label{est:sk}
s_k = 2k|\log \lambda| - \log\left(\frac{M^{1-p}\hbar^2}{p-1} \right)  + o(1),
\end{equation}
which follows $\lim_{k\to +\infty}\frac{s_{k-1}}{s_k} = 1$. Thus, it is reasonable to assume that $\frac{y_{k-1}^+}{\sqrt{s_{k-1}}}$ and $\frac{y_{k-1}^+}{\sqrt{s_{k}}}$ tend to the positive root $\zeta$ as $k \to +\infty$. Hence,
$$\alpha = \kappa^{-1} f \left(\frac{\zeta}{(M^{-1}\kappa)^{\frac{p-1}{2}} h_k \lambda^{-1}}\right) + o(1).$$
Using the definition \eqref{def:f} of $f$, we have 
$$\alpha = \left(1 + c_p\left|\frac{\zeta}{\left(M^{-1}\kappa\right)^{\frac{p-1}{2}}h_k} \right|^2\lambda^2 \right)^{-\frac{1}{p-1}} + o(1),$$
which follows
\begin{equation}\label{est:zeta}
\left|\frac{\zeta}{\left(M^{-1}\kappa\right)^{\frac{p-1}{2}}h_k} \right|^2 = \frac{1}{c_p}\left[(\alpha^{1-p} - 1)\lambda^{-2}\right] + o(1),
\end{equation}
where $c_p$ is the constant given in the definition \eqref{def:f} of $f$.\\
Substituting this into \eqref{equ:rel_uk_W} and using again the definition \eqref{def:f} of $f$, we arrive at 
\begin{align*}
v_k(z) &= M \left(1 + c_p\left|\frac{\zeta}{\left(M^{-1}\kappa\right)^{\frac{p-1}{2}}h_k} \right|^2 z^2 \right)^{-\frac{1}{p-1}} + o(1)\\
& = M \left(1 + (\alpha^{1-p} - 1)\lambda^{-2}z^2\right)^{-\frac{1}{p-1}} + o(1).
\end{align*}
Let $k \to +\infty$, we then obtain the conclusion $ii)$.\\

\noindent $iii)$ From \eqref{est:zeta} and the fact that $\frac{y_k^+}{\sqrt{s_k}} \to \zeta$ as $k \to +\infty$, we have 
$$(h_k^{-1}y_k^+)^2 = \frac{(\alpha^{1-p} - 1) M^{1-p}}{c_p\lambda^2(p-1)}\log s_k + o(1).$$
Using \eqref{est:sk}, we then derive 
$$(h_k^{-1}y_k^+)^2 = \frac{2k |\log \lambda| (\alpha^{1-p} - 1) M^{1-p}}{c_p\lambda^2(p-1)} - \frac{(\alpha^{1-p} - 1) M^{1-p}}{c_p\lambda^2(p-1)}\log\left(\frac{M^{1-p}\hbar^2}{p-1} \right)  + o(1),$$
which yields the conclusion $iii)$ and completes the proof of Theorem \ref{theo:3}.
\end{proof}

\subsection{The numerical results}\label{sec:4}
This subsection gives various numerical confirmation for the assertions stated in the previous subsection (Theorem \ref{theo:3}). All the experiments reported here used $\varphi(x) = 2(1 + cos(\pi x))$ as the initial data, $\alpha = 0.6$ as the parameter for controlling the interval to be refined, $\lambda = \frac{1}{2}$ as the refining factor, $C_\Delta = \frac 14$ as the stability condition for the scheme \eqref{equ:scheme}, $p=3$ and $a = 0.1, 1, 10$ in the nonlinearity $F$ given in \eqref{def:nolF}. The threshold $M$ is chosen to be satisfied the condition \eqref{rel:Mlambda}. In Table \ref{tab:1}, we give some values of $M$ corresponding the initial data and the initial space step $\hbar$.
\numberwithin{table}{section}
\begin{table}[!htbp]
\begin{center}
\tabcolsep = 8mm
\begin{tabular}{|l|llll|}
\hline
$\hbar$ & 0.040 & 0.020 & 0.010 & 0.005 \\
\hline
$M$ & 0.320 & 0.160 & 0.080 & 0.040 \\
\hline
\end{tabular}
\end{center} 
\caption{The value of $M$ corresponds to the initial data and the initial space step.}\label{tab:1}
\end{table}
We generally stop the computation after $40$ refining phases. Indeed, since $h_k^\frac{2}{p-1}\|u_k(\centerdot, \tau_k^*)\|_\infty = M$ and the fact that $h_k = \lambda h_{k-1}$, we have by induction,
$$\|u_k(\centerdot, \tau_k^*)\|_\infty = h_{k}^{-\frac{2}{p-1}}M = (\lambda h_{k-1})^{-\frac{2}{p-1}}M = \dots = (\lambda^k \hbar)^{-\frac{2}{p-1}}M.$$
With these parameters, we see that the corresponding amplitude of $u$ approaches $10^{12}$ after $40$ iterations.\\

\noindent $i)$ \textbf{The value $\frac{\tau^*_k}{\tau_{k}}$ is independent of $k$ and tends to the constant as $k \to +\infty$}.\\
It is convenient to denote the computed value of $\frac{\tau_k^*}{\tau_{k}}$ by $N_k$ and the predicted value given in the statement $i)$ of Theorem \ref{theo:3} by $N_{pre}$. Note that the values of $N_{pre}$ does not depend on $a$ but depend on $\hbar$ because of the relation \eqref{rel:Mlambda}. More precisely, 
$$N_{pre}(\hbar) = \frac{(1 - \lambda^2)\|\varphi\|_\infty^{1 - p}}{C_\Delta(p-1)\hbar^2}.$$
Then considering the quality $\frac{N_k}{N_{pre}}$, theoretically, it is expected to converge to $1$ as $k$ tends to infinity. Table \ref{tab:2} provides computed values of $\frac{N_k}{N_{pre}}$ at some selected indexes of $k$, for computing with $\hbar = 0.005$ and three different values of $a$. According to the numerical results given in Table \ref{tab:2}, the computed values in the case $a = 10$ and $a = 1.0$ approach to $1$ as expected which gives us a numerical answer for the statement \eqref{equ:blrate}. However the numerical results in the case $a = 0.1$ is not good due to the fact that the speech of convergence to the blow-up limit \eqref{equ:blrate} is $\frac{1}{|\log (T-t)|^{a'}}$ with $a' = \min\{a,1\}$ (see Theorem \ref{theo:refinedasymptotic}).

\begin{table}[!htbp]
\begin{center}
\tabcolsep = 8mm
\begin{tabular}{|c|c|c|c|}
\hline
$k$& $a = 10$ & $a = 1.0$ & $a = 0.1$ \\
\hline
10 & 1.0325 & 0.9699 & 0.5853\\
15 & 1.0203 & 0.9771 & 0.5885\\
20 & 1.0149 & 0.9816 & 0.5923\\
25 & 1.0117 & 0.9845 & 0.5957\\
30 & 1.0096 & 0.9867 & 0.5989\\
35 & 1.0080 & 0.9885 & 0.6016\\
40 & 1.0072 & 0.9899 & 0.6043\\
\hline
\end{tabular}
\end{center} 
\caption{The values of $\frac{N_k}{N_{pre}}$ at some selected indexes of $k$,  for computing with $\hbar = 0.005$ and three different values of $a$.}\label{tab:2}
\end{table}

\vspace*{0.3cm}
\noindent $ii)$ \textbf{The function $v_k(z)$ introduced in part $ii)$ of Theorem \ref{theo:3} converges to a predicted profile as $k \to +\infty$.}\\
As stated in part $ii)$ of Theorem \ref{theo:3}, if we plot $v_k(z)$ over the fixed interval $(-1,1)$, then the graph of $v_k$ would converge to the predicted one. Figure \ref{fig:1} gives us a numerical confirmation for this fact, for computing with $\hbar = 0.005$ and $a = 10$. Looking at Figure \ref{fig:1}, we see that the graph of $v_{k}$ evidently converges to the predicted one given in the right-hand side of \eqref{equ:vk} as $k$ increases. The last curve $v_{40}$ seemly coincides to the prediction. Figure \ref{fig:2} shows the graph of $v_{40}$ and the predicted profile for an other experiment with $\hbar = 0.005$ and $a = 0.1$. They coincide to within plotting resolution.
\begin{figure}[!htbp]
\begin{center}
\includegraphics[scale = 0.4]{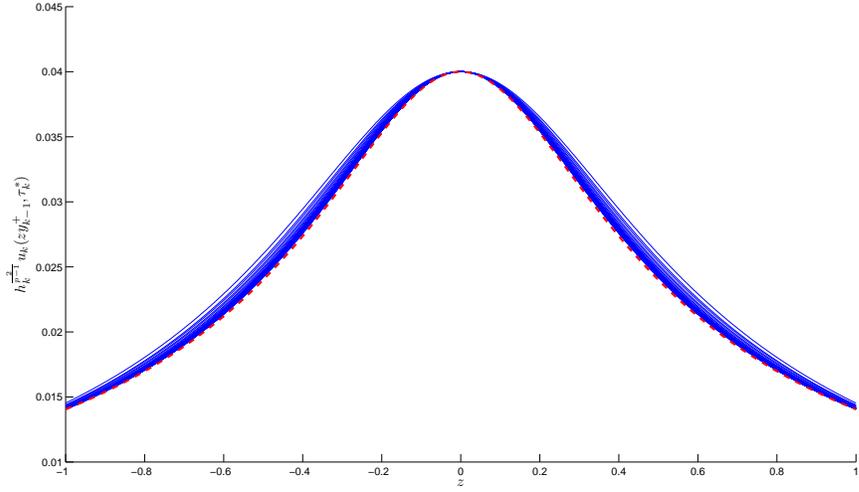}
\end{center}
\caption{The graph of $v_k(z)$ at some selected indexes of $k$, for computing with $\hbar = 0.005$ and $a = 10$. They converge to the predicted profile (the dash line) as stated in \eqref{equ:vk} as $k$ increases.}
\label{fig:1}
\end{figure}

\begin{figure}[!htbp]
\begin{center}
\includegraphics[scale = 0.4]{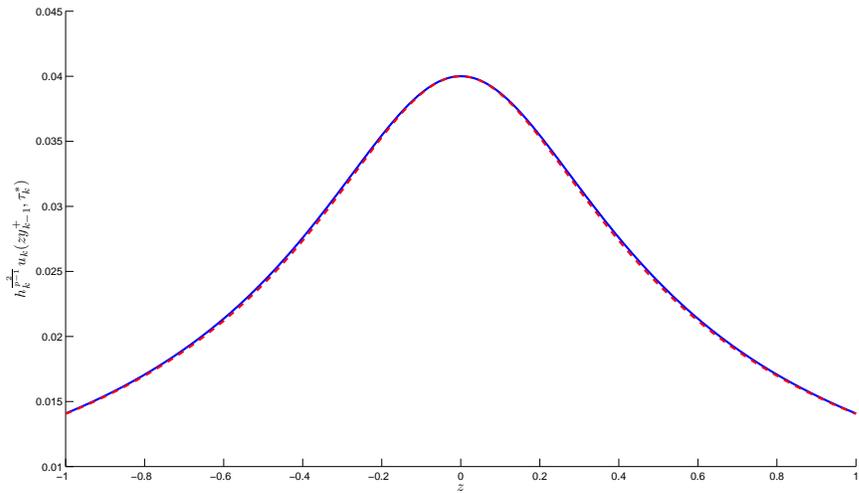}
\end{center}
\caption{The graph of $v_k(z)$ at $k = 40$ and the predicted profile given in \eqref{equ:vk}, for computing with $\hbar = 0.005$ and $a = 0.1$. They coincide within plotting resolution.}
\label{fig:2}
\end{figure}

In Table \ref{tab:3}, we give the error in $L^\infty$ between $v_k(z)$ at index $k = 40$ and the predicted profile given in the right hand-side of \eqref{equ:vk}, namely
\begin{equation}\label{def:eha}
e_{\hbar,a} = \sup_{z \in (-1,1)}\left|v_{40}(z) - M\left(1 + (\alpha^{1-p} - 1)\lambda^{-2}z^2\right)^{-\frac{1}{p-1}}\right|.
\end{equation}
These numerical computations give us a confirmation that the computed profiles $v_k$ converges to the predicted one. Since the error $e_{\hbar,a}$ tends to $0$ as $\hbar$ goes to zero, the numerical computations also answer to the stability of the blow-up profile stated in $i)$ of Theorem \ref{theo:pro}. In fact, the stability  makes the solution visible in numerical simulations.
\begin{table}[!htbp]
\begin{center}
\tabcolsep = 8mm
\begin{tabular}{|c|c|c|c|}
\hline
$\hbar$& $a = 10$ & $a = 1.0$ & $a = 0.1$ \\
\hline
0.04 & 0.002906 & 0.001769 & 0.002562\\
0.02 & 0.000789 & 0.000671 & 0.000687\\
0.01 & 0.000470 & 0.000359 & 0.000380\\
0.005 &0.000238 & 0.000213 & 0.000235\\
\hline
\end{tabular}
\end{center} 
\caption{Error in $L^\infty$ between the computed and predicted profiles, $e_{\hbar,a}$ defined in \eqref{def:eha}.}\label{tab:3}
\end{table}

\vspace*{0.5cm}
\noindent  $iii)$ \textbf{The quality $(h_k^{-1}y_k^+)^2$ behaves like a linear function.}\\
For making a quantitative comparison between our numerical results and the predicted behavior as stated in $iii)$ of Theorem \ref{theo:3}, we plot the graph of $(h_k^{-1}y_k^+)^2$ against $k$ and denote by $\gamma_{\hbar,a}$ the slope of this curve. Then considering the ratio $\frac{\gamma_{\hbar,a}}{\gamma}$, where $\gamma$ is given in part $iii)$ of Theorem \ref{theo:3}. As expected, this ratio $\frac{\gamma_{\hbar,a}}{\gamma}$ would approach one. Figure \ref{fig:3} shows $(h_k^{-1}y_k^+)^2$ as a function of $k$ for computing with the initial space step $\hbar = 0.005$ for different values of $a$. Looking at Figure \ref{fig:3}, we see that the two middle curves corresponding the case $a = 10$ and $a = 1$ behave like the predicted linear function (the top line), while this is not true in the case $a = 0.1$ (the bottom curve). In order to make this clearer, let us see Table \ref{tab:4} which lists the values of $\frac{\gamma_{\hbar,a}}{\gamma}$ for computing with various values of the initial space step $\hbar$ for three different values of $a$. Here, the value of $\gamma_{\hbar,a}$ is calculated for $20 \leq k \leq 40$. As Table \ref{tab:4} shows that the numerical values in the case $a = 10$ and $a = 1$ agree with the prediction stated in $ii)$ of Theorem \ref{theo:3}, while the numerical values in the case $a = 0.1$ is far from the predicted one. 
\begin{figure}[!htbp]
\begin{center}
\includegraphics[scale = 0.4]{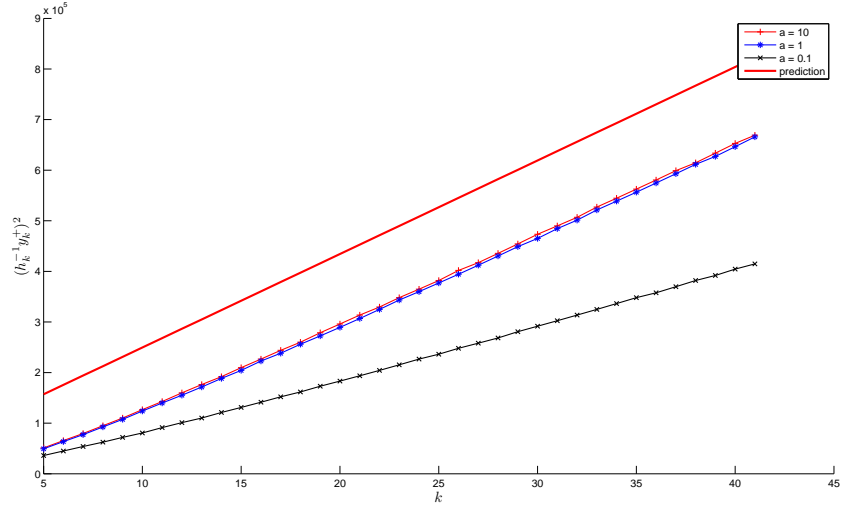}
\end{center}
\caption{The graph of $(h_k^{-1}y_k^+)^2$ against $k$, for computing with $\hbar = 0.005$ for three different values of $a$.}
\label{fig:3}
\end{figure}

\begin{table}[!htbp]
\begin{center}
\tabcolsep = 8mm
\begin{tabular}{|c|c|c|c|}
\hline
$\hbar$& $a = 10$ & $a = 1.0$ & $a = 0.1$ \\
\hline
0.04 & 1.9514 & 1.9863 & 1.9538\\
0.02 & 1.1541 & 1.1436 & 0.8108\\
0.01 & 0.9991 & 1.0052 & 0.6417\\
0.005 &0.9669 & 0.9682 & 0.5986\\
\hline
\end{tabular}
\end{center} 
\caption{The values of $\frac{\gamma_{\hbar,a}}{\gamma}$ for computing with various values of the initial space step $\hbar$ for three different values of $a$.}\label{tab:4}
\end{table}

\appendix
\renewcommand*{\thesection}{\Alph{section}}
\counterwithin{theo}{section}
\section{Appendix A}
\noindent The following lemma from \cite{NG14the} gives the expansion of $\phi(s)$, the unique solution of equation \eqref{equ:phiODE} satisfying \eqref{equ:solphi}:
\begin{lemm} \label{ap:lemmA3} Let $\phi$ be a positive solution of the following ordinary differential equation:
\begin{equation*}
\phi_s = -\frac{\phi}{p-1} + \phi^p + \frac{\mu \phi^p}{\log^a(2 + e^\frac{2s}{p-1}\phi^2)}.
\end{equation*}
Assuming in addition $\phi(s) \to \kappa$ as $s \to +\infty$, then $\phi(s)$ takes the following form:
\begin{equation*}
\phi(s) = \kappa(1 + \eta_a(s))^{-\frac{1}{p-1}} \quad \text{as} \quad s \to +\infty,
\end{equation*}
where
$$\eta_a(s) \sim C_*\int_s^{+\infty}\frac{e^{s-\tau}}{\tau^a}d\tau = \frac{C_*}{s^a}\left(1 + \sum_{j\geq 1}\frac{b_j}{s^j}\right),$$
with $C_* = \mu\left(\frac{p-1}{2}\right)^a$ and $b_j = (-1)^j\prod_{i = 0}^{j-1}(a+i)$.
\end{lemm}
\begin{proof} See Lemma A.3 in \cite{NG14the}.
\end{proof}
\section{Appendix B}
We aim at proving the following:
\begin{lemm}[\textbf{Estimate of $\omega(s)$}] \label{lemm:apB1} We have 
$$|\omega(s)| = \mathcal{O}\left(\frac{1}{s^{a + 1}}\right), \quad \text{as} \quad s \to +\infty.$$
\end{lemm}
\begin{proof} From Lemma \ref{ap:lemmA3}, we write 
$$p(\phi(s)^{p-1} - \kappa^{p-1}) = -\frac{p \eta_a(s)}{p-1}(1 + \eta_a(s))^{-1} = -\frac{pC_*}{(p-1)s^a}(1 + \eta_a(s))^{-1} + \mathcal{O}\left(\frac{1}{s^{a+1}}\right).$$
A direct calculation yields, 
\begin{align*}
e^{-s}h'\left(e^\frac{p}{p-1}\phi(s)\right) & = \frac{\mu p\phi^{p-1}(s)}{\log^a(2 + e^\frac{2s}{p-1}\phi^2(s))} - \frac{2a\mu e^\frac{2s}{p-1}\phi^{p+1}(s)}{(2 + e^\frac{2s}{p-1}\phi^2(s))\log^{a+1}(2 + e^\frac{2s}{p-1}\phi^2(s))}\\
& =\frac{pC_*}{(p-1)s^a}(1 + \eta_a(s))^{-1} + \mathcal{O}\left(\frac{1}{s^{a+1}}\right).
\end{align*}
Adding the two above estimates, we obtain the desired result. This ends the proof of Lemma \ref{lemm:apB1}.
\end{proof}
\begin{lemm}[\textbf{Estimate of $R(y,s)$}] \label{lemm:apB2} We have 
$$|R(y,s)| = \mathcal{O}\left(\frac{|y|^2 + 1}{s^{a' + 1}}\right), \quad \text{as} \quad s \to +\infty,$$
with $a' = \min\{1,a\}$.
\end{lemm}
\begin{proof} Let us write $\varphi(y,s) = \frac{\phi(s)}{\kappa}\nu(y,s)$ where 
$$\nu(y,s) = \kappa\left(1+ \frac{p-1}{4ps}\sum_{j=1}^l y_j^2\right)^{-\frac{1}{p-1}} + \frac{\kappa l}{2ps}.$$
Then, we write $R(y,s) = \frac{\phi(s)}{\kappa}R_1(y,s) + R_2(y,s)$ where 
\begin{align*}
R_1(y,s) &= \nu_s - \Delta \nu - \frac{y}{2}\cdot\nabla \nu - \frac{\nu}{p-1} + \nu^p,\\
R_2(y,s) &= -\frac{\phi'}{\kappa}\nu - \frac{\phi}{\kappa}\nu^p + \phi^p\left(\frac{\nu}{\kappa} \right)^p + e^{-\frac{ps}{p-1}}h'\left(e^\frac{s}{p-1}\frac{\phi \nu}{\kappa} \right).
\end{align*}
The term $R_1(y,s)$  is already treated in \cite{VELcpde92}, and it is bounded by 
$$|R_1(y,s)| \leq \frac{C(|y|^2 + 1)}{s^2} + C \mathbf{1}_{\{|y| \geq 2K_0 \sqrt{s}\}}.$$
To bound $R_2$, we use the fact that $\phi$ satisfies \eqref{equ:solphi} to write
\begin{align*}
R_2(y,s) &= \frac{\nu \phi}{\kappa^p}(\kappa^{p-1} - \phi^{p-1})(\kappa^{p-1} - \nu^{p-1})\\
& + e^{-\frac{ps}{p-1}}\left[h\left(e^\frac{s}{p-1}\frac{\phi \nu}{\kappa} \right) - h\left(e^\frac{s}{p-1} \phi \right)  \right]\\
& + \left( 1 - \frac{\nu}{\kappa}\right)e^{-\frac{ps}{p-1}}h\left(e^\frac{s}{p-1} \phi \right).
\end{align*}
Noting that $\nu(y,s) = \kappa + \bar{\nu}(y,s)$ with $|\bar{\nu}(y,s)|\leq \frac{C}{s}(|y|^2 + 1)$, uniformly for $y \in \mathbb{R}$ and $s \geq 1$, and recalling from Lemma \ref{ap:lemmA3} that $\phi(s) = \kappa(1 + \eta_a(s))^{-\frac{1}{p-1}}$ where $\eta_a(s) = \mathcal{O}(s^{-a})$, then using a Taylor expansion, we derive 
$$|R_2(y,s)| \leq C\left(\frac{|y|^2 + 1}{s^{a + 1}} +  \mathbf{1}_{\{|y| \geq 2K_0 \sqrt{s}\}}\right).$$
This concludes the proof of Lemma \ref{lemm:apB2}.
\end{proof}

\def\cprime{$'$}

\vspace*{0.4cm}
\noindent $\begin{array}{ll} \textbf{Address:} & \text{Universit\'e Paris 13, Institut Galil\'ee, LAGA,}\\
& \text{99 Avenue Jean-Baptiste Cl\'ement,}\\
& \text{93430 Villetaneuse, France.}\\
\end{array}$

\vspace*{0.2cm}
\noindent $\begin{array}{ll}\textbf{E-mail:}& \text{vtnguyen@math.univ-paris13.fr}\\
& \text{Hatem.Zaag@univ-paris13.fr}
\end{array}$

\end{document}